\documentclass[reqno]{amsart}[11pt] % reqno puts labels of equations on the right
\usepackage{subfigure}
\usepackage[margin=3.5cm]{geometry}
\usepackage{color}
\usepackage{graphicx}
\usepackage{enumerate}
\usepackage{xcolor}
\usepackage{mathtools,amsthm,amssymb, amsmath, amssymb, amsfonts, epsfig}
\usepackage{setspace, stmaryrd, verbatim, euro, enumerate}
\usepackage{bbm, tensor,comment,textcomp, todonotes}
\usepackage{bm}
\usepackage{appendix}
\usepackage{multirow}

\usepackage{colortbl}

\newtheorem{theorem}{Theorem}[section]

\newtheorem{lemma}[theorem]{Lemma}

\numberwithin{equation}{section}

\theoremstyle{definition}

\newtheorem{assumption}[theorem]{Assumption}

\newtheorem{remark}[theorem]{Remark}

\theoremstyle{plain}
\newtheorem{example}[theorem]{Example}

\definecolor{winered}{rgb}{0.6,0,0}
\definecolor{ocean}{rgb}{0,0.1,0.6}
\definecolor{forest}{rgb}{0,0.7,0.2}
\definecolor{fill}{rgb}{.92,1,.92}
\definecolor{sunset}{rgb}{.9,0.4,0}

\usepackage[colorlinks=true, linkcolor=ocean,citecolor=winered,urlcolor=winered]{hyperref}

\newcommand{\norm}[1]{\left\lVert#1\right\rVert}

\newcommand{\abs}[1]{\left\lvert#1\right\rvert}

\newcommand\cA{\mathcal{A}}
\newcommand\cB{\mathcal{B}}

\newcommand\cE{\mathcal{E}}
\newcommand\cF{\mathcal{F}}
\newcommand\cG{\mathcal{G}}
\newcommand\cH{\mathcal{H}}

\newcommand\cM{\mathcal{M}}

\newcommand\cS{\mathcal{S}}
\newcommand\cT{\mathcal{T}}

\newcommand\cX{\mathcal{X}}

\newcommand{\overbar}[1]{\mkern 3mu\overline{\mkern-3mu#1\mkern-1mu}\mkern 1mu}

\newcommand\bD{\mathbb{D}}
\newcommand\bE{\mathbb{E}}

\newcommand\bN{\mathbb{N}}
\newcommand\bP{\mathbb{P}}

\newcommand\bR{\mathbb{R}}
\newcommand\bT{\mathbb{T}}

\newcommand\TT{\mathbf{T}}

\newcommand{\D}{\mathrm{d}}

\newcommand{\DD}{\mathrm{D}}
\newcommand{\KK}{\mathrm{K}}
\newcommand{\MM}{\mathrm{M}}

\newcommand{\half}{\frac{1}{2}}

\newcommand{\ep}{\varepsilon} 
 
\newcommand{\notthis}[1]{}
\newcommand{\one}{\mathbbm{1}}

\newcommand{\dt}{\, \mathrm{d}t}
\newcommand{\ds}{\, \mathrm{d}s}
\newcommand{\du}{\, \mathrm{d}u}

\newcommand{\dx}{\, \mathrm{d}x}

\newcommand{\ttn}{\bm{t_n}}
\newcommand{\xxn}{\bm{x_n}}
\newcommand{\ttk}{\bm{t_k}}
\newcommand{\xxk}{\bm{x_k}}

\usepackage{charter} % font
\usepackage[foot]{amsaddr} % print address on first page instead of last

\title{Pathwise large deviations for white noise chaos expansions}
\author{Alexandre Pannier}
\address{Department of Mathematics, Imperial College London}
\email{a.pannier17@imperial.ac.uk}
\thanks{The author expresses his sincere gratitude to Antoine Jacquier for his guidance and support, and to Eyal Neuman and Amarjit Budhiraja for stimulating discussions.}

\begin{document}

%%%%%%%%%%%%%%%%%%%%%%%%%%%%%%%%%%%%%%%%%%
\begin{abstract}
We consider a family of continuous processes~$\{X^\varepsilon\}_{\varepsilon>0}$ which are measurable with respect to a white noise measure, take values in the space of continuous functions~$C([0,1]^d\colon\mathbb{R})$, and have the Wiener chaos expansion
    \[
    X^\varepsilon = \sum_{n=0}^{\infty} \varepsilon^n I_n \big(f_n^{\varepsilon} \big).
    \]
   We provide sufficient conditions for the large deviations principle of~$\{X^\varepsilon\}_{\varepsilon>0}$ to hold in~$C([0,1]^d\colon\mathbb{R})$, thereby refreshing a problem left open by P\'erez--Abreu (1993) in the Brownian motion case. The proof is based on the weak convergence approach to large deviations: it involves demonstrating the convergence in distribution of certain perturbations of the original process, and thus the main difficulties lie in analysing and controlling the perturbed multiple stochastic integrals. 
   Moreover, adopting this representation offers a new perspective on pathwise large deviations and induces a variety of applications thereof.
\end{abstract}

%\subjclass[2010]{60F10, 60G15, 60H05, 60H07}

\maketitle

\section{Introduction}

This paper contributes to the large body of work dedicated to limit theorems in relation with the Wiener chaos.
In a probability space~$(\Omega,\cF,\bP)$, any Gaussian functional~$F\in L^2$ can be represented via its Wiener chaos expansion
\begin{equation} \label{eq:WCE}
    F = \bE F + \sum_{n=1}^\infty I_n(f_n),
\end{equation}
where, in the setup we are interested in, the terms of this (orthogonal) decomposition are multiple Wiener--It\^o integrals and the kernels~$\{f_n\}_{n\ge1}$ are deterministic and uniquely determined by~$F$.  We refer to the monograph~\cite{Nualart06} for more details. 
By exposing the inner structure of Gaussian functionals in a systematic way, this route generated a fruitful field of asymptotic methods. 
In particular, the Fourth Moment Theorem of Nourdin and Peccati~\cite{NP05}, later combined with Stein's method~\cite{ET15,NP09Stein,NP12}, is an active field of research which studies the convergence of elements of a fixed Wiener chaos, i.e. multiple Wiener--It\^o integrals. %The subsequent papers~\cite{NP13,NNP13} examined the rates of convergence in total variation on a Wiener chaos. 
In the same line of research, the Breuer--Major theorem~\cite{BM83} was recently revisited by means of Malliavin calculus and Wiener chaos in infinite-dimensional spaces~\cite{CNN20,NZ20}.
These results and the related limit theorems from~\cite{AN19,BT20,NNP13} fall, for the most part, into the scope of central and noncentral limit theorems on Wiener chaos, and were successfully applied to stochastic partial differential equations~\cite{DNZ19,HNV20,NZ20}.

However, to the best of our knowledge, the literature on large deviations for such expansions has not evolved since the nineties.
A family $\{X^\ep\}_{\ep>0}$ of random variables taking values in some Polish space~$\cX$ satisfies a large deviations principle (LDP) with rate function $\Lambda\colon\cX\to[0,+\infty]$ if, for all Borel subsets $B\subset\cX$, the inequalities
\begin{align} \label{eq:LDPdef}
-\inf_{x\in B^\circ} \Lambda(x) \le 
\liminf_{\ep \downarrow 0} \ep \log \bP \big( X^\ep\in B \big) 
\le \limsup_{\ep \downarrow 0} \ep \log \bP \big( X^\ep\in B \big) 
\le -\inf_{x\in \overbar{B}} \Lambda(x)
\end{align}
hold, and the level sets $\{x\in\cX: \Lambda(x)\le M\}$ of~$\Lambda$ are compact for all $M>0$. 
The LDP for families of a fixed Wiener chaos, even Banach space-valued, is well-known from~\cite{DS89,Ledoux90} and was derived again in~\cite{MWNPA92} in the space of continuous functions. Around the same time, an LDP for a whole, real-valued, Wiener chaos expansion was obtained in~\cite{PAT94} in the Brownian motion case, see~\cite{PA93} for a review. 
Our main contribution, Theorem~\ref{thm:main}, reconciles the best of both worlds: we establish an LDP in the space of continuous functions \emph{and} for the whole Wiener chaos expansion. Moreover,  we consider the case of a general white noise measure. 

Among the many limit theorems on Wiener space cited above, very few are concerned with random variables on path space. The lift from the real-valued LDP of~\cite{PAT94} to our sample-path version parallels the transition from the Gärtner-Ellis theorem for finite-dimensional LDPs to the theory of Friedlin and Wentzell~\cite{Freidlin84}: both aim at studying the probability that the whole path of a stochastic process hits a given set. Furthermore, the pathwise LDP is stronger in the sense that an application of the contraction principle suffices to recover the finite-dimensional one.

The proof of~\cite{PAT94} relied on a type of exponentially good approximation theorem, which preserves the LDP with a change of the rate function. Although these techniques are nowadays ubiquitous in sample-path large deviations for e.g. solutions of stochastic equations, applying this approximation to a Wiener chaos expansion in the space of continuous functions is challenging since it requires a good estimate of
\[
\bP\bigg( \sup_{t\in[0,1]} \abs{I_n(f_n^t)} > \delta \bigg), \quad \mathrm{for\; all} \; n\ge1.
\]
%\red{Is that enough? is my assumption weaker?}
This highlights the difficulty of deriving bounds for such objects, in contrast with stochastic equations where one can exploit the local martingale structure of the stochastic integrals. It also exposes the limits of the standard approach based on exponential approximations, especially in infinite-dimensional spaces.
% talk about these papers
% how they are stuck with standard LDP techniques

Deviating from this approach, we will instead follow the methods presented in the monographs~\cite{DE97,BD19} which rely upon the equivalence between the LDP~\eqref{eq:LDPdef} and the so-called Laplace principle with same rate function~$\Lambda$: for all continuous bounded maps $H\colon\cX\to\bR$,
\begin{equation*}
    \lim_{\ep\downarrow0} - \ep \log \bE\left[\exp\left\{-\frac{H(X^\ep)}{\ep}\right\}\right] =  \inf_{x\in\cX}\big\{ \Lambda(x)+H(x)\big\}.
    \label{eq:Laplace}
\end{equation*}
This route introduced by Dupuis and Ellis was dubbed the weak convergence approach because it essentially boils down to proving the convergence in distribution of a family of processes related to~$\{X^\ep\}$. 
It enabled to prove LDPs for new types of Gaussian-driven systems such as multiscale diffusions~\cite{DS12,Morse17}, stochastic Volterra equations~\cite{JP20,Zhang08} and infinite-dimensional processes~\cite{BM09,BD01,BDM08,BDS19}.
Our choice of setup therefore lies at the intersection between the domains of application of the Wiener chaos expansion and the weak convergence approach. A conducive integrator is the white noise measure displayed in~\cite[Section 1.1.2]{Nualart06}, which, as we will see in the next section, covers a wide range of infinite-dimensional Gaussian processes. We note at this point that spatially homogeneous coloured noise and Poisson random measures also offer potential research directions. 
% weak convergence to the rescue
% at the intersection between WCE and WCA
% our choice and other possible directions

Regarding the choice of state space, we are only limited by the tightness criteria available to us.
We consider the space of continuous functions~$C(\KK\colon\bR)$, where~$\KK$ is a compact subset of~$\bR^d$. To the best of our knowledge, the latter condition is necessary to obtain a type of Kolmogorov continuity criterion; in this paper we will rely upon the multidimensional version derived in~\cite{Ibra83}. There is no apparent reason why the theory should not also apply to~$C(\KK\colon\bR^N)$ for~$N\ge2$ but tracking the indices becomes cumbersome so we restrict to the one-dimensional case for clarity. Nevertheless, our state space covers infinite-dimensional processes living in~$C \big ([0,1] \colon C([0,1]\colon\bR) \big)$ for instance. Relatively to the classical discretisation and approximation argument, the weak convergence approach has proved efficient in deriving LDPs in such spaces.
% infinite dim processes
% the difficulties we faced

In the framework of stochastic (partial differential) equations, pathwise uniqueness is a necessary and sometimes tricky condition, while in this paper the difficulty lies elsewhere. Here, we must consider the convergence in distribution of a family of processes where the stochastic integration takes place with respect to a \emph{shifted} white noise, such as~$W^u(\D x)=W(\D x) + \int u(x) \dx$. Hence, after replacing~$W(\D x)$ by~$W^u(\D x)$ in the multiple stochastic integrals, the terms of the decomposition~\eqref{eq:WCE} are not independent any longer, and developing~$W^{u}(\D x_1) \cdots W^{u}(\D x_n)$ generates~$2^n$ multiple integrals. We then classify and analyse these new integrals in order to obtain the necessary estimates.

We stress that the assumptions laid down in this paper only involve the kernels of the chaos expansion and are therefore purely deterministic. 
Moreover, in the context of pathwise large deviations, the chaos expansion is more agnostic than the stochastic equations paradigm, and hence this representation opens the LDP gates to uncommon continuous models. In particular, we provide applications to (multiple) Skorohod integrals and linear Skorohod equations driven by space-time white noise, and to Wick--It\^o integrals driven by fractional Brownian motion (fBm).
So far, LDPs for anticipating integrals have only been considered in the context of anticipating SDEs driven by a standard Brownian motion~\cite{MM98,MNS91a,MNS92}. Related asymptotic results include central limit theorems for multiple~\cite{NN10} and fBm-driven~\cite{BBN20,HN14} Skorohod integrals. For completeness, we also mention the LDPs for stochastic integrals derived in~\cite{Ganguly18,Garcia08}.

The outline of the paper is the following. Section~\ref{sec:prelis} introduces the white noise and the chaos expansion. In Section~\ref{sec:main}, we describe our model, assumptions, and main result. The proof is presented in Section~\ref{sec:proof}: first we adapt sufficient conditions for the LDP of Gaussian functionals to our framework, then we lay down the foundations of the integrals under consideration, and finally we split the proof of the conditions in different lemmas. Section~\ref{sec:applis} gathers our applications and we make some concluding remarks in Section~\ref{sec:ccl}.

%\textbf{Notations}: modulus of continuity; weak convergence *2; $L^2$ spaces

\section{Definitions and notations}
\label{sec:prelis}

\subsection{The white noise measure}
Let us fix a probability space~$(\Omega,\cF,\bP)$, a convex set~$\bT\subseteq[0,+\infty)$ that includes~$0$, and a sigma-finite measure~$\mu$ on a measurable space~$(\bm E,\cE)$. Then we define the parameter space~$\TT:=\bT\times \bm E$ and its~$\sigma$-algebra~$\cT:=\cB(\bT)\otimes \cE$. Following~\cite{Nualart06}, we define a white noise measure~$\dot W$ as an~$L^2(\Omega)$-Gaussian measure on~$(\TT,\cT)$ based on the (atomless) product measure~$\widehat \mu:=\lambda\otimes\mu$, where~$\lambda$ is the Lebesgue measure on~$\bT$. This means~$\dot W$ has zero mean and the following covariance structure for all~$A_1,A_2 \in \cT$,
\[
\bE\Big[ \dot W\big(A_1\big) \dot W\big(A_2\big) \Big] = \widehat\mu\big(A_1 \cap A_2 \big) .
\]
This exhibits the main feature of the white noise measure, namely that it takes independent values on any family of disjoint subsets of~$\TT$. 
The following examples demonstrate how one can recover well-known Gaussian processes from the white noise measure. For all~$n\ge1$, we will henceforth write~$L^2(\TT^n)$ to denote the Hilbert space~$L^2\big(\TT^n,\widehat\mu^{\,\otimes n} \big)$ of real-valued functions.

%%%%%%%%%%%%%%%%%%%%%%%%%%%%%%%%%%
\begin{example}[Space-time white noise]
\label{ex:BS}
In the case where~$(\bm E,\cE)= \big(\bR^m,\cB(\bR^m) \big)$, $m\ge1$, and~$\mu$ is the Lebesgue measure on~$\bR^m$, $\dot W$ is called space-time white noise.
When~$\mu$ is not the Lebesgue measure one speaks of spatially inhomogeneous white noise~\cite{Neuman18}.
A similar setup with~$\bm E=\bT^m$ also gives rise to the Brownian sheet~\cite[Example 3.10]{DKMNX08}. %Setting~$m=0$ reduces the latter to the standard Brownian motion case. 
\end{example}

\begin{example}[Isonormal process]
The theory of integration presented in~\cite{Nualart06} allows to define an isonormal process~$X$ over~$L^2(\TT)$ as the stochastic integral
\[
X(h) := \int_\TT h(t,x)  W(\D t \dx), \quad h\in L^2(\TT),
\]
where, following standard notations, we drop the dot when integrating. This generates the covariance structure~$\bE\big[ X(h) X(g) \big] = \langle h, g \rangle_{L^2(\TT)}$, for~$h,g\in L^2(\TT)$.

We can recover the cylindrical Brownian motion~$B$ defined in~\cite[Definition 11.3]{BD19} by setting~$B_t(\phi):= X \big( \one_{[0,t]}\, \phi \big)$, for all~$\phi\in L^2(\bm E)$ and~$t\ge0$.
Another variant consists in picking~$\phi\in C_c^\infty(\TT)$.
\end{example}

%%%%%%%%%%%%%%%%%%%%%%%%%%%%%%%%%%
\begin{example}[Multidimensional Brownian motion]
\label{ex:nDBM}
We reproduce Example 1.1.2 in~\cite{Nualart06}: let~$\bm E=\{1,\cdots,N\}$ for some~$N\ge1$ and~$\mu$ be the uniform measure which gives mass one to each point~$1, \cdots , N$.  Then~$W^i(t):=\dot W([0,t]\times\{i\})$ is a standard Brownian motion for all~$t\in\bT,\,i\in\bm E,$ with covariance~$\bE[W^i(t) W^j(s)]= (s\wedge t) \one_{i=j}$. For any~$h\in L^2(\TT)$, one can also define its stochastic integral, which coincides with the Itô integral:
\[
W(h) = \sum_{i=1}^N \int_\bT h^i(t) \, \D W^i(t).
\]
\end{example}

\begin{example}[Fractional Brownian motion] \label{ex:fbm}
For this example, we refer to~\cite[Section 5.1]{Nualart06} and set~$\TT=\bT$. The process~$W$ defined as~$W_t:=\dot W([0,t])$ is a standard Brownian motion, and there exists a deterministic function~$K_H\colon\TT^2\to\TT$ such that~$\{B^H_t\}_{t\in\TT}$ defined as
\begin{equation}\label{eq:fbmdef}
B^H_t := \int_0^t K_H(t,s) \,\D W_s
\end{equation}
is a fractional Brownian motion with Hurst exponent~$H\in(0,1)$ and covariance function
\[
R_H(t,s)=\half \big(t^{2H} + s^{2H} -\abs{t-s}^{2H}\big).
\]
\end{example}
%%%%%%%%%%%%%%%%%%%%%%%%%%%%%%%%%%%%%%%%%%%%%
\subsection{The Wiener chaos}\label{subsec:chaos}
Let~$\cF^W:=\sigma\big(\big\{\dot W(A):A\in\cT \big\}\big)$ denote the~$\sigma$-field generated by~$\dot W$; it is then clear that~$\cF^W$ coincides with the~$\sigma$-fields generated by the examples presented above. 
From now on, we will write~$L^2(\Omega)$ as short for~$L^2 \big(\Omega,\cF^W,\bP \big)$.
It is well-known that any random variable~$F\in L^2 (\Omega)$ has the Wiener chaos expansion~\eqref{eq:WCE},
where the convergence takes place in~$L^2(\Omega)$, see for instance~\cite[Theorem 1.1.2]{Nualart06}.
For all~$n\ge1$, and for the remainder of the paper, $f_n\in L^2(\TT^n)$ are symmetric functions and~$I_n\colon L^2(\TT^n)\to L^2(\Omega)$ is a continuous linear map that can be interpreted as a multiple integral:
\begin{equation}\label{eq:Indef}
I_n(f_n) := \int_{\TT^n} f_n(\ttn,\xxn) \,W(\D t_1 \, \D x_1)\cdots W(\D t_n\, \D x_n),
\end{equation}
with~$\ttn=(t_1,\cdots,t_n)$ and~$\xxn=(x_1,\cdots,x_n)$, where~$(t_i, x_i) \in\TT$ for all~$i\in \llbracket 1,n \rrbracket$. The~$n$th Wiener chaos~$\cH_n$ is the subset of~$L^2(\Omega)$ made of all the elements with the representation~\eqref{eq:Indef}.
The notion of integration is recalled in subsection~\ref{subsec:integral}, but we already highlight the important identity
\begin{equation} \label{eq:Inidentity}
    \norm{I_n(f_n)}_{L^2(\Omega)} = \sqrt{n!} \norm{f_n}_{L^2(\TT^n)}.
\end{equation}
In the next section we specify our framework and the main large deviations result. 
%%%%%%%%%%%%%%%%%%%%%%%%%%%%%%%%%%%%%%%%%%%%%%%%%%%%%%%

\section{Statement of the main result} \label{sec:main}
%%%%%%%%%%%%%%%%%%%%%%%%%%%%%
We aim at proving a large deviations principle for a family~$\{X^\ep\}_{\ep>0}$ of processes taking values in~$C(\KK\colon\bR)$, where~$\KK\subset\bR^d$ is a compact subspace equipped with the Euclidean norm denoted~$\norm{\cdot}$, and~$d\ge1$. For all~$\ep>0$ and~$z\in\KK$, we define these processes as:
\begin{equation}
X^\ep(z) := \sum_{n=0}^\infty \ep^n I_n(f_n^{\ep,z}).
\label{eq:Xep}
\end{equation}
For all~$n\ge1,\,\ep>0$ and~$z\in\KK$, the kernels~$f_n^{\ep,z}$ are symmetric functions belonging to~$L^2(\TT^n)$, $f_0^{\ep,z}=\bE[X^\ep(z)]$ and~$I_0$ denotes the identity mapping on constants. Since the~$\{I_n\}_{n\ge0}$ are independent, it is clear that~$X^\ep(z)\in L^2(\Omega)$ for all~$z\in\KK$.
%%%%%%%%%%%%%%%%%%%%%%%%%%%%%%%%%%%%%%%%%%%
We consider the following sufficient conditions, which in particular ensure the continuity of~$X^\ep$, as pointed out in Remark~\ref{rk:continuity}.
\begin{assumption}
\label{assu:main}\
\begin{enumerate}[(A)]
    \item There exist~$\{f_n^z\}_{n\ge0,z\in\KK}$, where~$f_n^z\in L^2(\TT^n)$, such that~$\lim_{\ep\downarrow0} f_0^{\ep,z} = f_0^z$ for all~$z\in\KK$ and
    \begin{align*}
       & \lim_{\ep\downarrow0}\,\sup_{z\in\KK}\, \sup_{n\ge1}\, \norm{f_n^{\ep,z} - f_n^z}_{L^2(\TT^n)} =0.
    \end{align*}
    \item For all~$\ep>0$ small enough, $z\in \KK$ and~$\kappa>0$ the following holds
    \begin{equation}
    \label{eq:seriesbound}
    \sum_{n=0}^\infty \sqrt{n!}\, \kappa^n \Big(\norm{f_n^{z}}_{L^2(\TT^n)} + \norm{f_n^{\ep,z}}_{L^2(\TT^n)} \Big) < \infty.
    \end{equation}
    \item  There exists a concave modulus of continuity~$\omega\colon[0,+\infty)\to[0,+\infty)$ satisfying, for some~\mbox{$\alpha_0>0$,}
\[
\int_0^1 \frac{\omega(s)}{s^{1+\alpha_0}} \ds<\infty \qquad \mathrm{and} \qquad \lim_{s\downarrow0} \frac{\omega(s)}{s} = \infty,
\]
    such that, for all~$\ep>0$ small enough, $y,z\in \mathrm{K}$ and~$\kappa>0$,
    \begin{equation}
    \label{eq:equicont}
        \sum_{n=0}^\infty \sqrt{n!}\, \kappa^n \Big(\norm{f_n^{z}-f_n^{y}}_{L^2(\TT^n)} + \norm{f_n^{\ep,z}-f_n^{\ep,y}}_{L^2(\TT^n)} \Big)
        \le \omega(\norm{z-y}).
    \end{equation}
\end{enumerate}

\end{assumption}

\begin{remark}\label{rk:assu} \
\begin{enumerate}[(A)]
    \item The first assumption is satisfied in the case where~$\bE[X^\ep(z)]$ converges as~$\ep$ goes to zero and~$\{f_n^{\ep,z}\}$ is independent of~$\ep$, which correspond to a purely small-noise problem. Example~\ref{ex:epdep} clarifies the need (or not) for the kernels to be~$\ep$-dependent.
    \smallskip
    \item The bound~\eqref{eq:seriesbound} holds in particular in the case where~$\{f_n^{z}\}$ and~$\{f_n^{\ep,z}\}$ are of {\em exponential type}, as defined in~\cite[Equation (3.2)]{PAT94}: there exists $\Delta>0$ such that for all~$z\in\KK$ and all~$n\ge0$ large enough,
    \begin{equation}
        \norm{f_n^{z}}_{L^2(\TT^n)} + \norm{f_n^{\ep,z}}_{L^2(\TT^n)} \le \frac{\Delta^n}{n!}.
    \label{eq:expotype}
    \end{equation}
    %Examples of such kernels arise from bilinear stochastic differential equations~\cite{LPA93}, and a simple case illustrates this in Example~\ref{ex:kernels}. \smallskip
    %To be more precise, the bound~\eqref{eq:seriesbound} also holds if we replace~$n!$ with~$(n!)^\gamma,\,\gamma>\half$ in equation~\eqref{eq:expotype}. 
    \item The condition~\eqref{eq:equicont} allows us to recover the equicontinuity necessary in Kolmogorov's continuity theorem and therefore the tightness of a family of processes related to~$\{X^{\ep}\}$. 
    The assumption can be satisfied with the modulus of continuity~$\omega:s\to s^\gamma$, $\gamma\in(0,1)$, and is also related to the ``slow-growth" assumption in~\cite{MWNPA92}, which yields the continuity of the map~$z\mapsto I_n(f_n^z)$.
    \item Various examples of processes satisfying all three assumptions are given in Section~\ref{sec:applis}.
\end{enumerate}
\end{remark}

%%%%%%%%%%%%%%%%%%%%%%%%%%%%%%%%%%%%%%%%%%%%%%%%%%%%%%%%%%%%%%%%%%%%%%%%%%%%%%%%%%%%%%%%%%%%%%%%%%%%%%%%%%%%%%%%%%%%%%%%%%
\medskip
For convenience, and despite a mild but common abuse of notations, we define the white noise process~$\{W_t\}_{t\in\bT}$ by~$W_t(E):= \dot W([0,t]\times E)$ for all~$E\in\cE$, along with its filtration~$\cF_t:=\sigma\big(\{ W_s(E): 0\le s\le t, E\in\cE\}\big)$. 
We think of~$W$ as a process when we need the concept of filtration but both points of view are equivalent. 
This allows us to introduce the set of stochastic controls
\begin{equation}\label{eq:cA}
\cA := \bigg\{ u\colon\Omega\times\TT\to\bR, \{\cF_t\} \mathrm{-progressively\; measurable, }\; \bE \left[\int_\TT u(t,x)^2 \,\widehat\mu(\D t \dx) \right] < \infty \bigg\},
\end{equation}
and, for all~$u\in\cA,\,z\in\KK$:
\begin{align}
\label{eq:Xu}
    \overbar X^u(z) := & \sum_{n=0}^\infty J_n^u (f_n^z), \quad \mathrm{where}\\ 
    J_n^u(f_n^z):= &\int_{\TT^n} f_n^z(\ttn,\xxn) \, \big( u(t_1,x_1) \,\widehat\mu(\D t_1 \dx_1) \big) \cdots \big( u(t_n,x_n) \,\widehat\mu(\D t_n \dx_n) \big).
\label{eq:Jnu}
\end{align}
\smallskip

The main result of the paper is the following.
%%%%%%%%%%%%%%%%%%%%%%%%%%%%%%%%%%%%%%%%%%%%%%%
\begin{theorem}
\label{thm:main}
Under Assumption~\ref{assu:main}, the family~$\{X^\ep\}_{\ep>0}$ defined in~\eqref{eq:Xep} satisfies an LDP in~$C(\KK)$ with rate function~$\Lambda\colon C(\KK) \to [0,+\infty]$ given by
\begin{equation}
    \Lambda(\psi) := \inf \bigg\{ \half \norm{u}_{L^2(\TT)}^2 : u\in L^2(\TT),\, \psi=\overbar X^u \bigg\},
\label{eq:ratefunction}
\end{equation}
and~$\Lambda(\psi)=+\infty$ if this set is empty.
\end{theorem}
%\medskip
\begin{remark}
The form of the rate function~\eqref{eq:ratefunction} is reminiscent of the rate function of the finite-dimensional LDP proved by P\'erez--Abreu and Tudor in~\cite[Theorem 3.1]{PAT94}.

In fact, for a fixed~$z\in\KK$, Assumption~\ref{assu:main} (A) and (B) alone are sufficient to derive an LDP for~$\{X^\ep(z)\}_{\ep>0}$ on~$\bR$, with rate function~$\Lambda_R\colon\bR\to[0,+\infty]$ given by
\[
\Lambda_R(r) := \inf \bigg\{ \half \norm{u}_{L^2(\TT)}^2 : u\in L^2(\TT),\, r=\overbar X^u(z) \bigg\}.
\]
\end{remark}
\begin{remark}
A variety of Gaussian processes can be expressed as~$g(W)$ where~$g$ is a linear and measurable map. As examples, we already mentioned the fractional Brownian motion (Example~\ref{ex:fbm}) and will soon introduce a infinite collection of Brownian motions~$\beta=\{\beta_i\}_{i\in\bN}$ (subsection~\ref{subsec:wca}). Furthermore, $\beta$ is itself connected through a measurable map to other infinite-dimensional Gaussian processes such as~$Q$-Wiener processes and cylindrical Brownian motions~\cite[Section 11.1]{BD19}. Hence the process~$\cG(\ep B)$, where~$B$ is one of the aforementioned Gaussian processes, can be written as~$(\cG\circ g)(\ep W)$ and be covered by our framework.
\end{remark}
%%%%%%%%%%%%%%%%%%%%%%%%%%%%%%%%%%%%%%%%%%%%%%%%%%%%%%%%%%%%%%%%%%%
The following example demonstrates the interest of the~$\ep$-dependence in the kernels.
\begin{example} \label{ex:epdep}
%Bilinear SDE to motivate the assumptions: expo type and small-time rescaling requires $f^\ep$. MDP too. Explicit kernels that satisfy the assumptions, can we say something about the process then? 
We take as an example the Brownian motion case with~$\KK=\TT=[0,1]$. For any~$\ep>0$, $x\in\bR$ and~$t\in\KK$, we consider the following small-noise SDE
\begin{equation}\label{eq:SDE}
X^\ep(t) = x + \int_0^t b( X^\ep (s)) \ds + \ep \int_0^t \sigma(X^\ep(s)) \,\D W_s,
\end{equation}
where~$b$ and~$\sigma$ are real-valued functions such that~\eqref{eq:SDE} has a unique strong solution.
Therefore, there exists a measurable map~$\cG^\ep\colon C[0,1]\to C[0,1]$ such that~$\cG^\ep(\ep W) := X^\ep$.
This yields an expansion for~$X^\ep$ in the form of~\eqref{eq:Xep}, where~$\ep^n$ arises from~$\ep W$. We notice that, for~$\ep_1\neq\ep_2$, the distributions of~$X^{\ep_1}$ and~$X^{\ep_2}$ are mutually singular, which means the map~$\cG^\ep$ can depend explicitly on~$\ep$, hence so do the kernels~$\{f_n^{\ep,t}\}$.
\end{example}

%%%%%%%%%%%%%%%%%%%%%%%%%%%%%%%%%%%%%%%%%%%%%%%%%%%%%%%%%%%%%%%%%%%%%%%%%%%%%%%%%%%%%%%%%%%%%%%%%%%%%%%%%%%%%%%%%%%%%%%%%%%%%%%%%%%%%%%%%%%%%%%%%%%%%%%%%%%%%%%%%%%%%%%%%%%%%%%%%%%%%%%%%%%%%%%%

%%%%%%%%%%%%%%%%%%%%%%%%%%%%%%%%%%%%%%%%%%%%%%%%%%%%%%%

\section{Proof of the main result} \label{sec:proof}
%%%%%%%%%%%%%%%%%%%%%%%%%%%%%%%%%%%%%%%%%%%%%%%%%%%%%%
The proof of Theorem~\ref{thm:main} consists of three parts. In the first one, we adapt the sufficient conditions for an LDP of Gaussian functionals, originally derived in~\cite{BD01}, to our setup. In the second, we give meaning to some multiple stochastic integrals arising from the first part and finally, we verify said conditions. The game is afoot.

%%%%%%%%%%%%%%%%%%%%%%%%%%%%%%%%%%%%%%%%%%%%%%%%%%%%%%%

\subsection{The weak convergence approach}\label{subsec:wca}
As mentioned in the introduction, we rely upon the weak convergence approach to large deviations which is conducive for proving LDPs of functionals of infinite-dimensional Brownian motions~\cite{BD01,BD19}. 
In order to deduce abstract sufficient conditions for the LDP as in~\cite[Theorem 11.13]{BD19}, we need to cast our white noise process in this framework.
This entails first expressing~$W=\{W_t(E)\}_{t\in\bT,E\in\cE}$ as a functional of an infinite sequence of standard Brownian motions. Building on this connection,
we then recover the variational representation formula for functionals of white noise, analogue of~\cite[Theorem 11.11]{BD19} as cornerstone of the proof. % point out that it also takes values in~$\cM(\TT)$. 

We start by observing that~$\dot W$ (and~$W$)
is a random variable taking values in the space of signed measures on~$\TT$, which we denote by~$\cM(\TT)$. Now, consider the Hilbert space~$L^2(\bm E):=\{f\colon\bm E\to\bR: \int_{\bm E} f(x)^2 \mu(\D x) <\infty\}$ equipped with the usual inner product, and let~$\{\phi_i\}_{i\in\bN}$ be a complete orthonormal sequence. Let~$\beta:=\{\beta_i\}_{i\in\bN}$ defined for all~$t\in\bT$ by~$\beta_i(t):=\int_{[0,t]\times \bm E} \phi_i(x) W(\D s\,\D x)$, and note that it is a sequence of independent standard Brownian motions. Furthermore, for all~$E\in\cE$, we have
\[
W_t(E) = \sum_{i=1}^\infty \beta_i(t) \int_{E} \phi_i(x) \mu(\D x),
\]
where the series converges in~$L^2(\Omega)$.
The sequence~$\beta$ can be regarded as a random variable with values in~$C([0,T]\colon\bR^\infty)$, where~$\bR^\infty$ denotes the product space of countably infinite copies of~$\bR$.
Therefore the $\sigma$-fields generated by $\beta$ and~$W$ (and~$\dot W$) coincide and there exists a measurable function~$g: C(\bT\colon\bR^\infty)\to \cM(\TT)$ such that~$\dot W=g(\beta)$ almost surely.

% \begin{remark} Funnily enough, when considering large deviations of the empirical measure of a Markov process~\cite[Chapter 6]{BD19}, one relies on a variational representation of a Markov chain and then build the empirical measure from it. Here, on the contrary, we have a variational representation for the random measure which we apply to a stochastic process. \end{remark}
We introduce the following spaces of deterministic and stochastic controls, and recall that~$\cA$ was defined by~\eqref{eq:cA}:
\begin{align*}
    \cS_M &:= \bigg\{ u\colon\TT\to\bR : \int_\TT u(t,x)^2 \,\widehat\mu(\D t \dx) \le M \bigg\}, \\
    \cA_M &:= \bigg\{ u\in\cA: u\in\cS_M, \;\bP\mathrm{-almost\; surely} \bigg\}, \qquad \mathrm{and}\qquad \cA_b := \bigcup_{M\in\bN} \cA_M,
\end{align*}
where $\cS_M$ is endowed with the weak topology on~$L^2(\TT)$, not to be confused with the weak convergence of measures. For all~$u\in\cA_b$, let us define $W^u = \big\{ W^u_t(E)\big\}_{t\in\bT, E\in\cE} $ by
\begin{equation}\label{eq:Wu}
W_t^u(E) := W_t(E) + \int_{[0,t]\times E} u(s,x) \,\widehat\mu(\D s \dx).
\end{equation}
We can now state the corresponding variational representation formula.
\begin{lemma}
For any bounded measurable map~$G\colon\cM(\TT)\to\bR$, the following holds:
\begin{equation}
    -\log \bE\left[ \mathrm{e}^{-G(W)}\right] = \inf_{u\in\cA_b} \bE\left[ \half \int_\TT u(t,x)^2 \,\widehat\mu(\D t \dx) + G(W^u) \right].
\end{equation}
\end{lemma}
\begin{proof}
The proof, based on the corresponding formula for~$\beta$ and its relationship with~$W$, is identical to the proof of~\cite[Theorem 11.11]{BD19} after replacing the Brownian sheet with the present white noise.
\end{proof}
From this lemma arise abstract sufficient conditions yielding an LDP for functionals of white noise. The functionals take the form~$\cG^\ep(\ep W)$, where~$\cG^\ep\colon\cM(\TT)\to \cX$ is a measurable map and~$\cX$ is a Polish space ($\cX=C(\KK)$ in the case we are interested in). With the tools that we just described, one can translate~\cite[Condition 11.12 and Theorem 11.13]{BD19} into the white noise framework, as follows.
\begin{assumption}
There exists a measurable map~$\cG^0\colon L^2(\TT)\to\cX$ such that the following hold.
\begin{enumerate}[(a)]
    \item Consider~$M>0$ and a family~$\{u^\ep\}_{\ep>0}\subset\cA_M$ such that~$u^\ep$ converges in distribution to~$u$ as~$\ep$ goes to zero. Then the following convergence also holds in distribution
    \[
    \lim_{\ep\downarrow0}\, \cG^\ep \left( \ep W^{u^\ep/\ep} \right) 
    = \cG^0 (u).
    \]
    \item For every~$M>0$, the following sets are compact subsets of~$\cX$:
    \begin{equation*}
    \label{eq:Gamma}
    \Gamma_M := \Big\{ \cG^0 (u) : u\in\cS_M \Big\}.
    \end{equation*}
\end{enumerate}
\label{assu:abstract}
\end{assumption}
\begin{remark}
Since~$\cS_M$ is compact for every~$M>0$ in the weak topology, Assumption~\ref{assu:abstract}(b) is satisfied as soon as~$\cG^0$ is continuous.
\end{remark}
\begin{theorem}
Let Assumption~\ref{assu:abstract} hold. Then~$\{\cG^\ep \left( \ep W \right)\}_{\ep>0}$ satisfies an LDP with rate function~$\Lambda\colon\cX\to [0,+\infty]$ given by
\begin{equation*}
    \Lambda(\psi) = \inf \bigg\{ \half \norm{u}^2_{L^2(\TT)} : u\in L^2(\TT),\, \psi=\cG^0(u) \bigg\},
\end{equation*}
and~$\Lambda(\psi)=+\infty$ if this set is empty.
\label{thm:abstract}
\end{theorem}

%%%%%%%%%%%%%%%%%%%%%%%%%%%%%%%%%%%%%%%%%%%%%%%%%%%%%%%%%%%%%%%%%%%%%%%%%%%%%%%%%%%%%%%%%%%%%%%%%%%%%%%%%%%%%%%
\subsection{Perturbing the chaos expansion} \label{subsec:integral}
We observe that, for each~$\ep>0$, there exist a measurable map~$\cG^\ep\colon\cM(\TT)\to C(\KK)$ such that, for all~$z\in\KK$,
\[
\cG^\ep(W)(z) =  \sum_{n=0}^\infty I_n(f_n^{\ep,z}),
\]
where the kernels come from~\eqref{eq:Xep}, and therefore we have~$X^\ep =\cG^\ep(\ep W)$ by linearity of~$I_n$.
Our objective is to apply the abstract Theorem~\ref{thm:abstract} to the map~$\cG^\ep$ defined above, hence we must study the controlled processes~$X^{\ep,u} := \cG^\ep \left(\ep W^{u/\ep} \right)$, for~$u\in\cA_b$.
Perturbing the chaos expansion leads us to revisit some notions of integration with respect to white noise.

% definitions
Let~$n\in\bN$. We recall how the multiple integral operator~$I_n$ is defined, following~\cite[Section 1.1.2]{Nualart06}. Let us call~$f\in L^2(\TT^n)$ an elementary function if it has the form
\begin{equation}\label{eq:eltry}
f(\ttn,\xxn) = \sum_{i_1,\cdots,i_n =1}^N a_{i_1 \cdots i_n} \mathbbm{1}_{A_{i_1}\times\cdots\times A_{i_n}}(\ttn,\xxn),
\end{equation}
where~$N\ge n$, $A_1,\cdots, A_N \in\cT$ are pairwise disjoint sets with finite measure, and the real coefficients~$a_{i_1,\cdots,i_n}$ are zero if and only if two of the indices~$i_1,\cdots,i_n$ are equal. For such a function, the integral is defined as
\begin{equation*}
I_n(f) = \sum_{i_1,\cdots,i_n =1}^N a_{i_1\cdots i_n} \dot W(A_{i_1})\times\cdots\times \dot W(A_{i_n}).
\end{equation*}
Let~$u\in\cA_b$, and let the shifted integral read 
\begin{equation}\label{eq:Imuelementary}
I_n^u(f) := \sum_{i_1,\cdots,i_n =1}^N a_{i_1 \cdots i_n} \dot W^u(A_{i_1}) \times\cdots\times \dot W^u(A_{i_n}),
\end{equation}
where, similarly to~\eqref{eq:Wu},
$\dot W^{u}(A) := \dot W(A) + \int_A u(t,x) \,\widehat \mu(\D t\dx)$ for all~$A\in\cT$, and such that~$\dot W^u$ is also~$\cM(\TT)$-valued. 

Let us define the random measure~$ \nu\colon A \to \int_{ A} u(t,x) \,\widehat \mu(\D t\dx)$, for~$A\in\cT$. In order to study~\eqref{eq:Imuelementary} and extend it to the whole space~$L^2(\TT^n)$, we develop~$\dot W^u(A_{i_1}) \times\cdots\times \dot W^u(A_{i_n})$ in full, which is suggestive of the binomial formula for~$(x+y)^n$.
We obtain the sum of~$2^n$ terms, each of them being the product of all the elements of an $n$-tuple. We can regard these~$n$-tuples as words of length~$n$ over an alphabet made of two distinct elements: the measures~$\nu$ and~$\dot W$. Moreover the order of these letters, encoded by the sets~$\{A_{i_j}\}_{j=1}^n$, matters.
Nevertheless, the~$2^n$ $n$-tuples can be classified in the following way.
For each~$k\in\llbracket0,n \rrbracket$ there are~$\binom{n}{k}$ terms made of exactly~$k$ times~$\nu$ and~$n-k$ times~$\dot W$; therefore let~$\Theta(k,n)$ be the set of such~$n$-tuples. An element~$\theta$ of~$\Theta(k,n)$ is thus a sequence made of~$\nu$ and~$\dot W$ and we can write, for a fixed sequence~$(i_1,\cdots,i_n)$:
\[
\prod_{j=1}^n \dot W^u(A_{i_j}) = \sum_{k=0}^n \sum_{\theta\in \Theta(k,n)} \prod_{j=1}^n \theta_j(A_{i_j}).
\]
For an elementary function~$f$ of the form~\eqref{eq:eltry} and a given~$n$-tuple~$\theta\in \Theta(k,n)$, we define
\[
m_\theta(f) := \sum_{i_1,\cdots,i_n =1}^N a_{i_1 \cdots i_n} \prod_{j=1}^n \theta_j(A_{i_j}).
\]
It is well-known that~$I_n$ can be extended to a linear map from~$ L^2(\TT^n)$ to~$L^2(\Omega)$. The purpose of the following lemma is to extend the maps~$m_\theta$ and~$I_n^u$ in a similar way and to identify this extension as a multiple integral.
\begin{lemma}\label{lemma:integration}
For all~$p\ge2$, $M>0,\,u\in\cA_M$ and~$n\ge0$, the following hold.
\begin{enumerate}
    \item For all~$0\le k\le n$ and~$\theta\in \Theta(k,n)$, the map~$m_\theta$ can be extended to a linear and continuous map from~$L^2(\TT^n)$ to~$L^p(\Omega)$
    with
    \begin{equation}\label{eq:boundmpf}
    \norm{m_\theta(f)}_{L^p(\Omega)} \le \sqrt{(n-k)!} \,M^{k/2} (p-1)^\frac{n-k}{2} \norm{f}_{L^2(\TT^n)}.
    \end{equation}
    \item The same extension holds for~$I^u_n$ with the bound
    \begin{equation}\label{eq:boundInu}
    \norm{I_n^u(f)}_{L^p(\Omega)} \le \sqrt{n!} \Big(4(M+1) (p-1)\Big)^{n/2} \norm{f}_{L^2(\TT^n)}.
    \end{equation}
    \item If~$\theta\in\Theta(n,n)$ and~$f\in L^2(\TT^n)$, then~$m_\theta(f)$ is also a multiple stochastic integral over~$\TT^n$ where integration against~$\nu$ takes place in the Lebesgue sense, almost surely. 
    Furthermore, we have~$m_\theta(f) = J^u_n(f)$, almost surely, where~$J^u_n$ is defined in~\eqref{eq:Jnu}.
\end{enumerate}
\end{lemma}

\begin{remark}
The bound~\eqref{eq:boundmpf} does not depend on~$\theta$ but only on~$k$ and~$n$.
\end{remark}

\begin{remark}\label{rem:IntegralInterpretation}
Like~$I_n$ in~\eqref{eq:Indef}, we can interpret~$m_\theta$ as an~$L^2(\Omega)$-multiple integral in the following sense:
    \begin{equation}\label{eq:mthetaintegral}
    m_\theta(f) = \int_{\TT^n} f(\ttn,\xxn) \, \nu(\D t_1 \, \D x_1)\cdots \nu(\D t_k \, \D x_k) W(\D t_{k+1} \,\D x_{k+1}) \cdots W(\D t_n\, \D x_n).
    \end{equation}
Likewise, we can interpret~$I_n^u$ as an~$L^2(\Omega)$-multiple integral in the following sense:
    \[
    I_n^{u} (f) = \int_{\TT^n} f_n(\ttn,\xxn) \, W^{u}(\D t_1 \dx_1) \cdots W^{u}(\D t_n \dx_n).
    \]
\end{remark}

\begin{proof}
Let us fix~$p\ge2$, $M>0,\,u\in\cA_M$ and~$n\ge0$.
\begin{enumerate}
    \item %\red{$m_\theta$ is linear for elementary functions.}
    Let~$0\le k\le n$ and~$\theta\in \Theta(k,n)$.   
    The map~$I_n$ is linear for elementary functions and therefore so is~$m_\theta$ since they are designed in an identical way.
    
    For all~$A\in\cT$ with finite measure, one has~$\nu(A) = \int_\TT \one_{A} \,\nu(\D t\dx) = \int_\TT \one_{A} \, u(t,x) \, \widehat\mu(\D t \dx)$. Without loss of generality, and for more clarity, assume~$\theta_j=\nu$ for all~$l\in\llbracket 1,k \rrbracket$ and~$\theta_j=\dot W$ for all~$l\in\llbracket k+1,n \rrbracket$. Otherwise, just split the~$i_j$'s in a different way in the computation below. Let~$f$ be an elementary function in~$L^2(\TT^n)$, then we have:
    %Now roll up your sleeves a little.
    \begin{align*}
     \hspace{-1.6cm}   m_\theta(f) 
        %& = (n-k)! \sum_{i_{k+1}<\cdots<i_n}^N\, \sum_{i_1,\cdots,i_k=1}^N a_{i_1 \cdots i_n} \prod_{j=1}^n \theta_j(A_{i_j}) \\
        & = \sum_{i_{k+1},\cdots,i_n=1}^N\, \sum_{i_1,\cdots,i_k=1}^N a_{i_1 \cdots i_n} \left(\prod_{j=1}^k \int_\TT \mathbbm{1}_{A_{i_j}}(t,x) \nu(\D t \dx) \right) \left(\prod_{j=k+1}^n \dot W(A_{i_j})\right) \\
        & = \sum_{i_{k+1},\cdots,i_n=1}^N\,  \left( \int_{\TT^k} \sum_{i_1,\cdots,i_k=1}^N a_{i_1 \cdots i_n} \left(\prod_{j=1}^k \mathbbm{1}_{A_{i_j}}(t_j,x_j) \right) \, \nu^{\otimes k}(\D \ttk \D \xxk) \right) \left(\prod_{j=k+1}^n \dot W(A_{i_j})\right) \\
        & \le  \sum_{i_{k+1},\cdots,i_n=1}^N\,  M^{k/2} 
        \left( \int_{\TT^k} \left( \sum_{i_1,\cdots,i_k=1}^N a_{i_1 \cdots i_n} \prod_{j=1}^k \mathbbm{1}_{A_{i_j}}(t_j,x_j) \right)^2 \, \widehat\mu^{\,\otimes k}(\D \ttk \D \xxk) \right)^\half \left(\prod_{j=k+1}^n \dot W(A_{i_j})\right),
    \end{align*}
    almost surely, where we used the linearity of the integral, Cauchy--Schwarz inequality~$k$~times and the fact that~$u\in\cA_M$, which yields an almost sure bound for its $L^2(\TT)$-norm. Now we observe that, since~$A_i$'s are disjoint, the cross terms vanish:
    \begin{align*}
    \Psi &:= \int_{\TT^k} \left( \sum_{i_1,\cdots,i_k=1}^N a_{i_1 \cdots i_n} \prod_{j=1}^k \mathbbm{1}_{A_{i_j}}(t_j,x_j) \right)^2 \, \widehat\mu^{\,\otimes k}(\D\xxk \D \ttk) \\
    & =  \int_{\TT^k} \sum_{i_1,\cdots,i_k=1}^N \left( a_{i_1 \cdots i_n}\right)^2 \left(\prod_{j=1}^k \mathbbm{1}_{A_{i_j}}(t_j,x_j) \right) \, \widehat\mu^{\,\otimes k}(\D\xxk \D \ttk) \\
    & = \sum_{i_1,\cdots,i_k=1}^N \left( a_{i_1 \cdots i_n}\right)^2 \prod_{j=1}^k \widehat\mu(A_{i_j}).
    \end{align*}
    We observe that, for any fixed~$(i_1,\cdots,i_k)$,
    \begin{equation}\label{eq:h}
    h:=\sum_{i_{k+1},\cdots,i_n=1}^N\sqrt{\Psi}\prod_{j=k+1}^n \one_{A_{i_j}} %= \norm{f}_{L^2(\TT^k}) %true but confusing at this point
    \end{equation}
    is an elementary function in~$L^2(\TT^{n-k})$, and therefore~$m_\theta(f) \le M^{k/2} I_{n-k}(h)$, almost surely.
    We then recover an element of the~$(n-k)$th Wiener chaos and appeal to the hypercontractivity property~\cite[Theorem 1.4.1]{Nualart06}. It implies that~$L^p(\Omega)$ norms are equivalent in a given Wiener chaos for all~$p\ge2$ and, in particular, that for all~$p\ge2$:
\[
\norm{I_{n-k}(h)}_{L^p(\Omega)} \le (p-1)^\frac{n-k}{2} \norm{I_{n-k}(h)}_{L^2(\Omega)} = (p-1)^\frac{n-k}{2} \sqrt{(n-k)!} \norm{h}_{L^2(\TT^{n-k})},
\]
where we also used~\eqref{eq:Inidentity}. Once more, the cross terms vanish to give
\[
\norm{h}_{L^2(\TT^{n-k})}^2 = \sum_{i_1,\cdots,i_n=1}^N \left( a_{i_1 \cdots i_n}\right)^2 \prod_{j=1}^n \widehat\mu(A_{i_j})
=\norm{f}_{L^2(\TT^{n})}^2,
\]
which yields the desired bound and thus the continuity of~$m_\theta$ for elementary functions. Moreover, the space of elementary functions is dense in~$L^2(\TT^n)$ and therefore~$m_\theta$ and its bound extend as claimed~\cite[Section 1.1.2]{Nualart06}.
\item Let~$f$ be an elementary function. This is a corollary of the previous item which follows by noticing that
\[
I^u_n (f) = \sum_{k=0}^n \sum_{\theta\in \Theta(k,n)} m_\theta(f).
\]
Then, bounding~$k$ and~$n-k$ by~$n$ and and using the identity~$\sum_{k=0}^n \binom{n}{k} = 2^n$, this yields
\begin{align*}
\norm{I^u_n(f)}_{L^p(\Omega)} &\le \sum_{k=0}^n \binom{n}{k} \sqrt{(n-k)!} \,M^{k/2} (p-1)^\frac{n-k}{2} \norm{f}_{L^2(\TT^n)}\\
&\le \sqrt{n!} \Big(4(M+1) (p-1)\Big)^{n/2} \norm{f}_{L^2(\TT^n)},
\end{align*}
as desired.
    \item Let~$f\in L^2(\TT^n)$ and a sequence of elementary functions~$\{f^l\}_{l\in\bN}$ tending to~$f$ in~$L^2(\TT^n)$. Let~$\theta\in\Theta(n,n)$, from~\eqref{eq:boundmpf}, we know that $m_\theta(f-f^l)$ converges to zero in~$L^2(\Omega)$. %, which means integration extends in the same way as~$I_n$ in~\eqref{eq:Indef}.
    Moreover, for every subsequence of~$\{m_\theta(f-f^l)\}_{l\in\bN}$, there exists a further subsubsequence that tends to zero almost surely. Therefore, the original sequence converges to zero almost surely, which implies that integration against~$\nu$ takes place in the Lebesgue sense, almost surely.
    
    Therefore, it is clear that, almost surely, we have
    \begin{align*}
    m_\theta(f) 
    &= \int_{\TT^n} f(\ttn,\xxn) \, \nu(\D t_1 \dx_1)  \cdots\nu(\D t_n \dx_n)\\
    &= \int_{\TT^n} f(\ttn,\xxn) \, \big( u(t_1,x_1) \,\widehat\mu(\D t_1 \dx_1) \big) \cdots \big( u(t_n,x_n) \,\widehat\mu(\D t_n \dx_n) \big)
    = J^u_n(f),
    \end{align*}
    as claimed.\qedhere
\end{enumerate}
\end{proof}
\begin{remark}
We observe that~$h$, defined in~\eqref{eq:h}, is equal to~$\norm{f(\bm{t_{n-k}},\bm{x_{n-k}},\,\bullet\,)}_{L^2(\TT^k)}$.
\end{remark}

\begin{remark}
It is natural to wonder whether representing~$m_\theta$ in the sense of Walsh~\cite{Walsh86,DKMNX08} would yield similar bounds. Because of the necessity of adaptedness, we would have to write~\eqref{eq:mthetaintegral} as:
\begin{equation*}
m_\theta(f_n) = n! \int_\TT \int_{[0,t_n]\times\bm E} \cdots \int_{[0,t_2]\times\bm E} f(\ttn,\xxn) \, \nu(\D t_1 \, \D x_1)\cdots \nu(\D t_k \, \D x_k) W(\D t_{k+1} \,\D x_{k+1}) \cdots W(\D t_n\, \D x_n),
\end{equation*}
with~$\theta\in\Theta(k,n)$ and~$f\in L^2(\TT^n)$. Applying Cauchy--Schwarz inequality~$k$ times yields
\begin{align*}
\bE \big[m_\theta(f)^2 \big] 
& \le (n!)^2 M^{k} \bE \left[\left( \int_\TT \int_{[0,t_n]\times\bm E} \cdots \int_{[0,t_{k+1}]\times\bm E} \norm{f}_{L^2(\TT^{n-k})} W(\D t_{k+1} \,\D x_{k+1}) \cdots W(\D t_n\, \D x_n) \right)^2 \right] \\
& = \left(\frac{n!}{(n-k)!}\right)^2 M^k \norm{f}_{L^2(\TT^{n})},
\end{align*}
which is not as sharp as~\eqref{eq:boundmpf}.
\end{remark}

%%%%%%%%%%%%%%%%%%%%%%%%%%%%%%%%%%%%%%%%%%%%%%%

\subsection{Proof of convergence}

Now that we have set the stage, the game consists in verifying that Assumption~\ref{assu:abstract} holds and then applying Theorem~\ref{thm:abstract}. 
In this subsection we will then study, for all~$z\in\KK$,
\[
X^{\ep,u}(z) =  \sum_{n=0}^\infty I_n^{\ep,u} (f_n^{\ep,z}),
\]
where~$I^{\ep,u}_0 =I_0$ and for all~$n\ge1$, $f_n\in L^2(\TT^n)$,
\begin{equation}
    I_n^{\ep,u} (f_n) = \int_{\TT^n} f_n(\ttn,\xxn) \;\ep W^{u/\ep}(\D t_1 \dx_1) \cdots \ep W^{u/\ep}(\D t_n \dx_n),
\label{eq:Inepu}
\end{equation}
which meaning flows from Lemma~\ref{lemma:integration} and Remark~\ref{rem:IntegralInterpretation}. Note that in this subsection, the $n$-tuples~$\theta$ are sequences made of the measures~$\ep\dot W$ (instead of~$\dot W$) and~$\nu$, but we retain the same notations for simplicity. Hence, for~$f\in L^2(\TT^n),\,\theta\in\Theta(k,n)$ and~$u\in\cA_M$, the linearity of~$m_\theta$ entails that the moment bound~\eqref{eq:boundmpf} becomes
\begin{equation}\label{eq:boundmep}
\norm{m_\theta(f)}_{L^p(\Omega)} \le \ep^{n-k} \sqrt{(n-k)!} \,M^{k/2} (p-1)^\frac{n-k}{2} \norm{f}_{L^2(\TT^n)}.
\end{equation}
We start by deriving a moment bound for~$X^{\ep,u}(z)$ in Lemma~\ref{lemma:bound}, then the tightness of~$\{ X^{\ep,u^\ep}\}_{\ep>0}$ in~$C(\KK)$ in Lemma~\ref{lemma:tightness}, which, combined with the
convergence of the marginals proved in Lemma~\ref{lemma:cvg}, yields the weak convergence in~$C(\KK)$. Finally we show the compactness of the sets~$\Gamma_M$, for all~$M>0$, in Lemma~\ref{lemma:goodness}.

\begin{lemma}[Moment bound]
Let Assumption~\ref{assu:main}(B) hold.
For all~$p\ge2,\,u\in\cA_M,M>0$ and~$\ep>0$ small enough, we have:~$\sup_{z\in\KK} \norm{X^{\ep,u}(z)}_{L^p(\Omega)} <\infty.$
\label{lemma:bound}
\end{lemma}

\begin{proof}
Let us fix~$\ep>0,\, u\in\cA_M,\,M>0$ and~$z\in\KK$, and recall the sequence of kernels~$\{f_n^{\ep,z}\}_{n\ge0}$ from~\eqref{eq:Xep}. As in the proof of Lemma~\ref{lemma:integration}, the representation
\[
I^{\ep,u}_n (f_n^{\ep,z}) = \sum_{k=0}^n \sum_{\theta\in \Theta(k,n)} m_\theta(f_n^{\ep,z}),
\]
combined with~\eqref{eq:boundmep}, yields
\begin{align*}
\norm{I^{\ep,u}_n (f_n^{\ep,z})}_{L^p(\Omega)} 
&\le \sum_{k=0}^n \sum_{\theta\in \Theta(k,n)} \norm{m_\theta(f_n^{\ep,z})}_{L^p(\Omega)}\\
&\le \sum_{k=0}^n \binom{n}{k} \ep^{n-k} \sqrt{(n-k)!} \,M^{k/2} (p-1)^\frac{n-k}{2} \norm{f_n^{\ep,z}}_{L^2(\TT^n)}.
% can we do better than lifting k to n?
\end{align*}
Hence, bounding~$k$ and~$n-k$ by~$n$, $\ep$ by one and using the identity~$\sum_{k=0}^n \binom{n}{k} = 2^n$, we conclude that 
\begin{align*}
    \norm{X^{\ep,u}(z)}_{L^p(\Omega)} 
    \le \sum_{n=0}^\infty \norm{I^{\ep,u}_n (f_n^{\ep,z})}_{L^p(\Omega)}
    \le \sum_{n=0}^\infty \sqrt{n!} \Big(4(M+1) (p-1)\Big)^{n/2}  \norm{f_n^{\ep,z}}_{L^2(\TT^n)},
\end{align*}
which is finite by Assumption~\ref{assu:main}(B).
\end{proof}

%%%%%%%%%%%%%%%%%%%%%%%%%%%%%%%%%%%%%%%%%
\begin{lemma}[Tightness]
\label{lemma:tightness}
Assume that Assumption~\ref{assu:main} (B) and (C) hold, let $M>0$ and a family~$\{u^\ep\}_{\ep>0}\subset\cA_M$, then the family~$\{ X^{\ep,u^\ep}\}_{\ep>0}$ is tight in~$C(\KK)$.
\end{lemma}

\begin{proof}
Let~$y,z\in\KK$ and~$\ep>0$. Using the same computations as in the proof of Lemma~\ref{lemma:bound} and Assumption~\ref{assu:main}(C), one obtains that for all~$p>2$:
\[
\norm{ X^{\ep,u^\ep}(z) - X^{\ep,u^\ep}(y)}_{L^p(\Omega)} \le  \sum_{n=0}^\infty  \sqrt{n!} \, \Big(4 (M+1) (p-1) \Big)^{n/2} \norm{f_n^{\ep,z} - f_n^{\ep,y}}_{L^2(\TT^n)}
\le \omega\big(\norm{ z-y} \big).
\]
We appeal to~\cite[Theorem 6]{Ibra83}, a multidimensional version of Kolmogorov's continuity theorem, which then states that for all~$\alpha>0$ and some constant~$C>0$:
\[
\bE \left[ \sup_{\norm{z-y}\le1} \frac{\abs{X^{\ep,u^\ep}(z) - X^{\ep,u^\ep}(y) }}{\norm{z-y}^\alpha}\right]
    \le C \int_0^1 \frac{\omega(u)}{u^{1+\alpha+d/p}} \du.
\]
The latter is finite by choosing~$\alpha$ and~$p$ such that~$\alpha+d/p<\alpha_0$, where we recall that~$\alpha_0>0$ comes from Assumption~\ref{assu:main}(C) and~$d$ is such that~$\KK\subset\bR^d$. This is possible because we can choose~$\alpha>0$ as small and~$p$ as large as we want, for instance~$\alpha<\alpha_0$ and~$p=1+d/(\alpha_0-\alpha)$.
Combining this with Lemma~\ref{lemma:bound}, Aldous theorem~\cite[Theorem 7.3]{Billingsey99} yields the tightness of~$\{ X^{\ep,u^\ep}\}_{\ep>0}$ in~$C(\KK)$.
\end{proof}
\begin{remark}\label{rk:continuity}
Suppose there exists a family of kernels~$\{f_n^z\}_{n\ge0,z\in\KK}$ such that~$f_n^z\in L^2(\TT^n)$ for all~$n\ge0$, and which satisfy conditions~\eqref{eq:seriesbound} and~\eqref{eq:equicont}. Define the real-valued process~$Z$ via the Wiener chaos expansion 
\[
Z(z) := \sum_{n=0}^\infty I_n(f_n^z),\quad z\in\KK.
\]
Then the version of the Kolmogorov's continuity theorem used above entails that~$Z$ admits a version which is H\"older continuous of any order~$\alpha<\alpha_0.$
\end{remark}

%%%%%%%%%%%%%%%%%%%%%%%%%%%%%%%%%%%%%%%%%%%%%%%%%

\begin{lemma}[Convergence]
Let Assumption~\ref{assu:main} (A) and (B) hold.
Consider~$M>0$ and a family~$\{u^\ep\}_{\ep>0}\subset\cA_M$ such that~$u^\ep$ converges in distribution to~$u$ as~$\ep$ goes to zero. Then, for all~$z\in\KK$, $X^{\ep,u^\ep}(z)$ converges in distribution to~$\overbar X^u(z)$, defined in~\eqref{eq:Xu}, as~$\ep$ goes to zero.
\label{lemma:cvg}
\end{lemma}
%%%%%%%%%%%%%%%%%%%%%%%%%%%%%%%%%%%%%%%%%%%%%%%%
\begin{proof}
We start by recalling, as Lemma~\ref{lemma:integration}(3) shows, that the multiple integral~$J^u_n$ defined in~\eqref{eq:Jnu} arises from the unique $n$-tuple belonging to~$\Theta(n,n)$, made exclusively of~$\nu$. On the other hand, all the $n$-tuples containing one or more~$\ep \dot W$ will converge to zero.

Let us fix~$z\in\KK$. We decompose the quantity of interest in the following way:
\[
X^{\ep,u^\ep}(z) - \overbar X^u(z)
=  \bm S_1^\ep + \bm S_2^\ep, 
\]
where
\begin{align*}
    \bm S^\ep_1 :=  \sum_{n=0}^\infty  \sum_{k=0}^{n-1}  \sum_{\theta\in \Theta(k,n)}  m_\theta(f_n^{\ep,z})
    \qquad \mathrm{and} \qquad
    \bm S^\ep_2 := \sum_{n=0}^\infty \sum_{\theta\in \Theta(n,n)} m_\theta(f_n^{\ep,z}) - \overbar X^u(z).
\end{align*}
Similarly to the proof of Lemma~\ref{lemma:bound}, the bound~\eqref{eq:boundmep} and the identity~$\sum_{k=0}^n \binom{n}{k} = 2^n$ yield
\begingroup
\addtolength{\jot}{.7em}
\begin{align*}
    \norm{\bm S_1^\ep}_{L^2(\Omega)} 
    \le  \sum_{n=0}^\infty \sum_{k=0}^{n-1} \sum_{\theta\in \Theta(k,n)} \norm{m_\theta(f_n^{\ep,z})}_{L^2(\Omega)}
    & \le  \sum_{n=0}^\infty \sum_{k=0}^{n-1} \binom{n}{k} M^{k/2} \ep^{n-k} \sqrt{(n-k)!} \norm{f_n^{\ep,z}}_{L^2(\TT^n)} \\
    & \le \ep \, \sum_{n=0}^\infty \sqrt{n!} \, \big(4 (M+1) \big)^{n/2}  \norm{f_n^{\ep,z}}_{L^2(\TT^n)},
\end{align*}
\endgroup
for all~$\ep<1$.
Assumption~\ref{assu:main}(B) entails that the series is finite and therefore~$\norm{\bm S_1^\ep}_{L^2(\Omega)}$ tends to zero as~$\ep$ goes to zero.

For all~$n\in\bN$, let us define
\[
\mathfrak J_n^{\ep,z} := J_n^{u^\ep}(f^{\ep,z}_n) - J_n^{u}(f^{z}_n), 
\qquad \mathrm{such that} \qquad
\bm S_2^\ep
= f_0^{\ep,z}-f_0^{z} + \sum_{n=1}^\infty \mathfrak J_n^{\ep,z}.
\]
According to the Skorohod representation theorem~\cite[Theorem A.3.9]{DE97}, there exists a probability space on which are defined~$\{\widetilde u^\ep\}_{\ep>0}$ and~$\widetilde u$ having the same law as~$\{u^\ep\}_{\ep>0}$ and~$u$ on the original probability space, and such that~$\widetilde u^\ep$ converges towards~$\widetilde u$ in an almost sure sense. This is a standard trick in the weak convergence literature; we retain the original labels and continue with almost sure convergence. However, we recall that this convergence takes place in the weak topology in~$L^2(\TT)$, and therefore we need to work out the convergence of~$\mathfrak J_n^{\ep,z}$ by induction.

Let us fix~$n\in\bN$ for the moment and recall that~$J^u_n$, defined in~\eqref{eq:Jnu}, maps~$L^2(\TT^n)$ to~$\bR$. We claim that for all~$k\in \llbracket 1,n-1 \rrbracket$, the following function, mapping~$\TT$ to~$\bR$, tends to zero in~$L^2(\TT)$, almost surely and for almost all~$(\bm{t_{n-k-1}},\bm{x_{n-k-1}})\in\TT^{n-k-1}$:
\begin{equation}\label{eq:induction}
J^{u^\ep}_k \big(f_n^{\ep,z}(\bm{t_{n-k-1}},\bm{x_{n-k-1}},\,\bullet\,) \big) - J^{u}_k \big(f_n^{z}(\bm{t_{n-k-1}},\bm{x_{n-k-1}},\,\bullet\,) \big).
\end{equation}
We name this claim~$\mathfrak C(k)$.

With this in mind, we start by pointing out a uniform bound. Assumption~\eqref{assu:main}(A) entails that, for~$\ep>0$ small enough, $\norm{f_n^{\ep,z}}_{L^2(\TT^n)} \le 2\norm{f_n^z}_{L^2(\TT^n)}$ for all~$n\ge0$ and~$z\in\KK$. Hence, by using Cauchy--Schwarz inequality~$k$ times, the following bound stands almost surely, for all~$k\le n$ and all~$(\bm{t_{n-k}},\bm{x_{n-k}})\in\TT^{n-k}$:
\begin{align} \label{eq:domination}
    \abs{J^{u^\ep}_k \big(f_n^{\ep,z}(\bm{t_{n-k}},\bm{x_{n-k}},\,\bullet\,) \big)
    - J^{u}_k \big(f_n^{z}(\bm{t_{n-k}},\bm{x_{n-k}},\,\bullet\,) \big)} 
    \le 3 M^{k/2} \norm{f_n^z (\bm{t_{n-k}},\bm{x_{n-k}},\,\bullet\,)}_{L^2(\TT^k)}.
\end{align}
We recall that if~$\{a_\ep\}$ and~$\{b_\ep\}$ are sequences in a Hilbert space~$(H, \,\langle \cdot,\cdot \rangle)$ such that $a_\ep\to a$ (strongly) and~$b_\ep\rightharpoonup b$ (weakly) as~$\ep$ goes to zero, then~$\langle a_\ep, b_\ep \rangle \to \langle a,b \rangle$; we call this property~$\mathfrak {P}(\rightharpoonup)$ for further reference.

Starting with~$k=1$, we know that~$\lim_{\ep\downarrow0}\norm{f_n^{\ep,z} - f_n^z}_{L^2(\TT^n)}=0$, which implies that for almost every~$(\bm{t_{n-1}},\bm{x_{n-1}})\in\TT^{n-1}$,
\begingroup
\addtolength{\jot}{1em}
\[
\lim_{\ep\downarrow0}\norm{f_n^{\ep,z}(\bm{t_{n-1}},\bm{x_{n-1}},\,\bullet\,) - f_n^{z}(\bm{t_{n-1}},\bm{x_{n-1}},\,\bullet\,)}_{L^2(\TT)}=0,
\]
\endgroup
and we also know that~$u^\ep \rightharpoonup u$, almost surely.
Property~$\mathfrak {P}(\rightharpoonup)$ entails that, almost surely and for almost every~$(\bm{t_{n-1}},\bm{x_{n-1}})\in\TT^{n-1}$, the following goes to zero as~$\ep$ goes to zero:
\begingroup
\addtolength{\jot}{.7em}
\begin{align}
\label{eq:Jcvg}
 J^{u^\ep}_1 \big(f_n^{\ep,z}&(\bm{t_{n-1}},\bm{x_{n-1}},\,\bullet\,) \big) - J^{u}_1 \big(f_n^{z}(\bm{t_{n-1}},\bm{x_{n-1}},\,\bullet\,) \big)\\
&  = \big\langle f_n^{\ep,z}(\bm{t_{n-1}},\bm{x_{n-1}},\,\bullet\,), u^\ep \big\rangle_{L^2(\TT)} - \big\langle f_n^{z}(\bm{t_{n-1}},\bm{x_{n-1}},\,\bullet\,), u \big\rangle_{L^2(\TT)}. \nonumber
\end{align}
\endgroup
Therefore, dominated convergence, induced by~\eqref{eq:domination}, yields that~$\mathfrak C(1)$ holds.

Now suppose that~$ \mathfrak C(k-1)$ holds, for some~$k\in \llbracket 2,n-1 \rrbracket$. We observe that, for any~$f\in L^2(\TT^{k})$,
\[
J^u_{k} (f) = \int_\TT J^u_{k-1} \big(f(t_{1}, x_{1},\,\bullet\,)\big) \, u(t_{1}, x_{1}) \, \widehat\mu(\D t_{1} \dx_{1})
=\langle J^u_{k-1}(f),u \rangle_{L^2(\TT)},
\]
which helps us see the following equality:
\begingroup
\addtolength{\jot}{.7em}
\begin{align*}
 J^{u^\ep}_{k} \big(f_n^{\ep,z}&(\bm{t_{n-k}},\bm{x_{n-k}},\,\bullet\,) \big) - J^{u}_{k} \big(f_n^{z}(\bm{t_{n-k}},\bm{x_{n-k}},\,\bullet\,) \big) \\
    & = \big\langle J^{u^\ep}_{k-1} \big(f_n^{\ep,z}(\bm{t_{n-k}},\bm{x_{n-k}},\,\bullet\,) \big), u^\ep \big\rangle_{L^2(\TT)}
    - \big\langle J^{u}_{k-1} \big(f_n^{z}(\bm{t_{n-k}},\bm{x_{n-k}},\,\bullet\,) \big), u \big\rangle_{L^2(\TT)}.
    %& = \int_\TT \Big[ J^{u^\ep}_k \big(f_n^{\ep,z}(\bm{t_{n-k}},\bm{x_{n-k}},\,\bullet\,) \big) u^\ep(t_{n-k},x_{n-k}) - J^{u}_k \big(f_n^{z}(\bm{t_{n-k}},\bm{x_{n-k}},\,\bullet\,) \big) u(t_{n-k},x_{n-k}) \Big]\, \widehat\mu(\D t_{n-k} \dx_{n-k}). 
\end{align*}
\endgroup
From~$\mathfrak {P}(\rightharpoonup)$ and the induction hypothesis~$\mathfrak C(k-1)$, the quantity above tends to zero almost surely and for almost every~$(\bm{t_{n-k-1}},\bm{x_{n-k-1}})\in\TT^{n-k-1}$.
By dominated convergence (from~\eqref{eq:domination}), we deduce that~$\mathfrak C(k)$ holds, and therefore the claim is proved for all~$k\in \llbracket 1,n-1 \rrbracket$. Combining~$\mathfrak C(n-1)$ with~$\mathfrak {P}(\rightharpoonup)$ again yields that~$\mathfrak J_n^{\ep,z}=J_n^{u^\ep}(f^{\ep,z}_n) - J_n^{u}(f^{z}_n)$ tends to zero almost surely as~$\ep$ goes to zero.

To conclude this proof we observe that, from~\eqref{eq:domination}, $\mathfrak J^{\ep,z}_n$ is dominated in the following sense
\[
\abs{\mathfrak J^{\ep,z}_n} 
\le 3 M^{n/2}\norm{f_n^z}_{L^2(\TT^n)} =: \mathfrak J_n^z,
\]
and that
$\sum_{n=1}^\infty\abs{\mathfrak J_n^z} < \infty$, by Assumption~\ref{assu:main}(B).
This entails, by dominated convergence, the almost sure convergence:
\[
\lim_{\ep\downarrow0} \bm S_2^\ep = \lim_{\ep\downarrow0} \left(  f^{\ep,z}_0 - f^{z}_0 + \sum_{n=1}^\infty \mathfrak J_n^{\ep,z} \right) = \sum_{n=1}^\infty \lim_{\ep\downarrow0} \mathfrak J_n^{\ep,z} = 0.
\]
Coming back from Skorohod's world, we in fact proved that, as~$\ep$ goes to zero, $\sum_{n=0}^\infty m_\theta(f_n^{\ep,z})$ converges towards~$\overbar X^u(z)$ in distribution in the original probability space, where~$\theta\in \Theta(n,n)$. We also showed that~$\bm S_1^\ep$ tends to zero in probability. Put together, they imply that~$\bm S_1^\ep+\sum_{n=0}^\infty m_\theta(f_n^{\ep,z})$ converges towards~$\overbar X^u(z)$ in distribution, as claimed.
\end{proof}
%%%%%%%%%%%%%%%%%%%%%%%%%%%%%%%%%%%%%%%%%%%%%%%%%%%%%%%%%%

%%%%%%%%%%%%%%%%%%%%%%%%%%%%%%%%%%%%%%%%ù
\begin{lemma}[Compactness]
\label{lemma:goodness}
Under Assumption~\ref{assu:main}, the following sets are compact subsets of~$C(\KK)$ for all~$M>0$:
\[
\Gamma_M = \Big\{ \overbar X^u : u\in\cS_M \Big\}.
\]
\end{lemma}

\begin{proof}
The proofs of relative compactness and closure parallel the proofs of tightness and convergence of the marginals, only in a deterministic and therefore simpler setting.  Let us fix~$M>0$ and consider a sequence~$\{X^{u_n}\}_{n\ge1} \subset \Gamma_M$. Similar calculations as in the proof of Lemma~\ref{lemma:bound} yield
\[
\sup_{z\in\KK,\,n\ge1} \abs{X^{u_n}(z)} < \infty.
\]
Furthermore, the equicontinuity follows as in the proof of Lemma~\ref{lemma:tightness}:
\[
\abs{ X^{u_n}(z) -X^{u_n}(y) } \le \omega\big( \norm{z-y} \big).
\]
Therefore an application of Arzel\'a--Ascoli's theorem yields the relative compactness of~$\Gamma_M$.

%\paragraph{Closure}
Consider a converging sequence~$\{X^{u_n}\}_{n\ge1} \subset \Gamma_M$ and denote its limit by~$X^\infty$.
Since~$\cS_M$ is compact with respect to the weak topology, there exists a converging subsequence~$\{u_{n_k}\}_{k\ge1}$ of~$\{u_n\}$, with limit~$u\in\cS_M$.
Then using the arguments of convergence of~$\bm S_2^\ep$ in the proof of Lemma~\ref{lemma:cvg}, one can prove that for all~$z\in\KK$,
\[
\lim_{k\uparrow\infty} \abs{X^{u_{n_k}}(z) - \overbar X^u(z)} =0.
\]
    Hence, we deduce that~$X^\infty=\overbar X^u$, which implies~$X^\infty\in\Gamma_M$. This entails that~$\Gamma_M$ is closed and thus compact.
\end{proof}

%%%%%%%%%%%%%%%%%%%%%%%%%%%%%%%%%%%%
\begin{proof}[Proof of Theorem~\ref{thm:main}]
Lemmas~\ref{lemma:tightness} and~\ref{lemma:cvg} respectively show the tightness of~$\{ X^{\ep,u^\ep}\}_{\ep>0}$ in~$C(\KK)$ and the convergence of its finite-dimensional distributions to those of~$\overbar X^u$. Together they yield the weak convergence condition of Assumption~\ref{assu:abstract}(a). Lemma~\ref{lemma:goodness} takes care of the other condition that needs to be satisfied for Theorem~\ref{thm:abstract} to hold. Therefore a direct application of the latter yields the desired claim.
\end{proof}

%%%%%%%%%%
%%%%%%%%%%%%%%%%%%%%%%%%%%%%%%%%%%%%%%%%%%%%%%%%%%%%%%%%%%

%%%%%%%%%%%%%%%%%%%%%%%%%%%%%%%%%%%%%%%%%%%%%%%%%%%%%
%%%%%%%%%%%%%%%%%%%%%%%%%%%%%%%%%%%%%%%%%%%%%%%%%%%%%

\section{Applications} \label{sec:applis}

The kernels of a random variable's chaos expansion are sometimes delicate to compute explicitly, making Assumption~\ref{assu:main} difficult to verify, but there exist interesting cases meeting this requirement. In this section, we exhibit families of processes which are covered by our framework along with the sufficient conditions for the pathwise LDP to hold.

We start by displaying in subsection~\ref{subsec:malliavin} a formula to compute the kernels of infinitely Malliavin differentiable processes. The following two subsections introduce a route to apply our main result (Theorem~\ref{thm:main}) to multiple Skorohod integrals and linear Skorohod equations driven by a white noise measure. Large deviations for anticipating equations were derived by exploiting their flow property in~\cite{MM98,MNS92} and~\cite{MNS91a}, for Stratonovich and Skorohod integrals respectively. However we note that these results are limited to the case of Brownian motion and single integral.

The Wick--It\^o integral is a conducive way of defining stochastic integrals driven by a fractional Brownian motion; we present a way of computing the kernels of such an integral in subsection~\ref{subsec:wick}. This also leads to sufficient conditions for an LDP to hold. Integrating against the fractional Brownian motion is a tricky business: large deviations were studied recently in the case~$H>\half$~\cite{BS20}, while our approach covers the whole range~$H\in(0,1)$. Finally, subsection~\ref{subsec:martingale} introduces a class of processes inspired from closed martingales in~$\bR_+$ where the continuity criterion (Assumption~\ref{assu:main}(C)) is simplified.

For clarity of exposition, we will replace~$(t,x)$ by~$t$, even where~$\TT$ represents time-space. By convention, an infimum over an empty set is equal to~$+\infty$. 
%%%%%%%%%%%%%%%%%%%%%%%%%%%%%%%%%%%%%%%%%%%%%%%%%%%%%%%%%%%%%%%%%%%%%%%%%%%%%%%%%%%%%%%
%\red{form of the rate function depends on Girsanov}
\subsection{Malliavin derivatives}\label{subsec:malliavin}
%This application relies upon the following crucial observation.
Adopting notations and definitions from~\cite[Section 1.2]{Nualart06}, we denote by~$D$ the Malliavin derivative operator and by~$\bD^{1,2}$ its domain of application in~$L^2(\Omega)$. For all~$n\ge1$, we also consider the iterated derivatives~$\DD^n$ and their domain of application~$\bD^{n,2}$, which are dense in~$L^2(\Omega)$. If~$F\in\bD^{\infty,2}:=\bigcap_{n\in\bN} \bD^{n,2}$, then the kernels of its chaos expansion are given explicitly by (see for instance~\cite[Exercise 1.2.6 ]{Nualart06}):
\begin{equation}\label{eq:kernelmalliavin}
f_n = \frac{1}{n!}\; \bE[ \DD^n F].
\end{equation}
This yields a verifiable condition for Assumption~\ref{assu:main} to hold, provided one can compute the~$L^2(\TT^n)$-norm of the kernels.
We note that in the white noise case, the Malliavin derivative~$\DD F$ is a stochastic process denoted~$\{ \DD_{t} F: t\in\TT\}$, and similarly~$\DD^n F$ can be written~$\{ \DD_{\ttn}^n F: \ttn\in\TT^n\}$.

\begin{example}
Let~$\KK$ and~$\TT$ be as general as possible and consider for all~$\ep>0$ and $z\in\KK$ the small-noise process
\[
X^\ep(z) = \exp\left(\ep \int_{\TT} h^z(t) \, W(\D t) - \half \ep^2\norm{h^z}^2_{L^2(\TT)}  \right),
\]
where~$h^z\in L^2(\TT)$ for all~$z\in\KK$ and~$\sup_{z\in\KK}\norm{h^z}_{L^2(\TT)}<\infty$. Notice that~$\TT$ is not necessarily equal to~$\KK$, hence the question of adaptation to a filtration is here meaningless.
Straightforward computations yield~$\bE[X^\ep(z)]=1$ and~$\DD^n_{\ttn} X^\ep(z) = \ep^n X^\ep(z) \prod_{i=1}^n h^z(t_i) $ for any~$\ttn\in\TT^n$. These give us the form of the kernels of the chaos expansion of~$X^\ep$ by~\eqref{eq:kernelmalliavin}:
\[
f_n^{\ep,z} (\ttn)
=\frac{\ep^n}{n!} \prod_{i=1}^n h^z(t_i), \quad \text{for all  } n\in\bN,\,\ep>0,\,z\in\KK,\ttn\in\TT^n,
\]
which shows that $f^{\ep,z}_0$ is independent of~$\ep$ and hence Assumption~\ref{assu:main}(A) holds. So does Assumption~\ref{assu:main}(B) since we have
\[
\norm{f_n^{\ep,z}}_{L^2(\TT^n)}
=\frac{\ep^n}{n!}  \norm{h^z}^n_{L^2(\TT)}.
\]
From the observation that $\prod_{i=1}^n a_i - \prod_{i=1}^n b_i = \sum_{i=1}^n \Big( (a_i-b_i) \prod_{j\neq i} c_j \Big)$, where $c$ is either $a$ or $b$, we can derive
\begin{align*}
\norm{f_n^{\ep,z}-f_n^{\ep,y}}_{L^2(\TT^n)} 
%&\le \frac{\ep^n}{n!}\norm{ \sum_{i=1}^n \big(h^z(t_i)-h^y(t_i)\big) \max_{j\neq i} \big\{h^z(t_j),h^y(t_j)\big\}^{n-1} }_{L^2(\TT^n)}\\
&\le \frac{\ep^n}{n!} n \norm{h^z-h^y}_{L^2(\TT)}    \max\big\{\norm{h^z}_{L^2(\TT)},\norm{h^y}_{L^2(\TT)}\big\}^{n-1}.
\end{align*}
This entails Assumption~\ref{assu:main}(C) holds as soon as~$\norm{h^z-h^y}_{L^2(\TT)}\le \omega(\norm{z-y})$ for an appropriate modulus of continuity~$\omega$.
\end{example}

\begin{example}
Let~$\KK=\TT=[0,1]$ such that we recover the Freidlin-Wentzell setting of Example~\ref{ex:epdep}, and consider the small-noise SDE
\[
X^\ep(t) = x + \ep \int_0^t \sigma(X^\ep(s))\,\D W_s, \qquad \text{for all  } t\in\TT,\,\ep>0,
\]
where~$x\in\bR$, $\sigma\in C^\infty(\bR)$ has linear growth and~$W$ is a standard Brownian motion. We ignore the drift for simplicity.

Let us consider the following induction hypothesis: for a given~$n\ge1$ and for all~$p\ge1$, there exists~$C_p<\infty$ such that~$\sup_{t\in\TT} \bE\left[\big\lvert\DD^k_{\ttk} X^\ep(t)\big\lvert^p\right] \le C_p^k$ for all~$k\le n$. 
This is clearly true for~$n=0$ by a standard argument based on Grönwall and Burkholder-Davis-Gundy inequalities. Now, noting that~$\DD\sigma(X)=\DD X\sigma'(X)$, using Hölder's inequality and~$\sigma\in C^\infty(\bR)$, one can conclude that a similar uniform bound holds for~$\DD^n_{\ttn}\sigma(X^\ep(t))$, with a possibly different constant than~$C_p$.
Then one can show that
\begin{align}\label{eq:DDn}
\DD^{n+1}_{\bm{t_{n+1}}} X^\ep(t) 
= \ep \int_0^t \DD^{n+1}_{\bm{t_{n+1}}} \sigma(X^\ep(s))\,\D W_s + \ep \sum_{i=1}^{n+1} \DD^{n}_{\widehat{t_n}(i)} \sigma(X^\ep(t_i)) \one_{t_i \le t}
%& = \ep \int_0^t \DD^{n+1}_{\bm{t_{n+1}}} \sigma(X^\ep(s))\,\D W_s + \DD^{n}_{\widehat{t_i^n}} \sigma(X^\ep(t_i)) \one_{t_i \le t}
\end{align}
where~$\widehat{t_n}(i):=(t_1,\cdots,t_{i-1},t_{i+1},\cdots,t_{n+1})$ for all~$n\ge0$. %and~$i^\ast$ is such that~$t_{i^\ast} \ge t_k$ for all~$k\le n+1$.
The terms in the sum are clearly bounded in~$L^p(\Omega)$ and since~$X^\ep$ is adapted they are non zero only if~$t_i  \ge t_k$ for all~$k\le n+1$.
Then using again~$\DD\sigma(X)=\DD X\sigma'(X)$, Hölder's inequality and~$\sigma\in C^\infty(\bR)$ we deduce that there exists yet another constant~$C_p$ such that
\[
\bE\left[ \Big\lvert \int_0^t \DD^{n+1}_{\bm{t_{n+1}}} \sigma(X^\ep(s))\,\D W_s\Big\lvert^p\right] \le C_p^n\left(1 + \int_0^t \DD^{n+1}_{\bm{t_{n+1}}} X^\ep(s)\,\ds \right)
\]
Grönwall's inequality thus entails that the induction hypothesis holds at~$n+1$, with a possibly different constant. Therefore the kernels~$\{f_n^{\ep,t}\}$ are well-defined by~\eqref{eq:kernelmalliavin}. Clearly~$f_0^{\ep,t}=x$ and moreover, for~$0\le s\le t\le 1$, some constant~$C$ and all~$n\ge1$, we have by~\eqref{eq:DDn}:
\begin{alignat*}{2}
& \norm{\bE\big[\DD^n_{\ttn} X^\ep(t) \big]}^2_{L^2(\TT^n)}\le n\sum_{i=1}^n \int_{\TT^n}&& \bE\left[ \Big\lvert\DD_{\widehat{t_{n-1}}(i)}^{n-1} \sigma(X^\ep(t_i))\Big\lvert^2\right] \,\D\ttn \le C^n; \\
& \norm{\bE \big[\DD^n_{\ttn} \big(X^\ep(t)-X^\ep(s)\big) \big]}^2_{L^2(\TT^n)} && \!\!\!\!\!\le n\sum_{i=1}^n \int_s^t \int_{\TT^{n-1}} \bE\left[ \Big\lvert\DD_{\widehat{t_{n-1}}(i)}^{n-1} \sigma(X^\ep(t_i))\Big\lvert^2\right] \,\D \widehat{t_{n}}(i) \,\D t_i\\
& && \!\!\!\!\!\le C^n (t-s).%^{n-1} \le C^n \sqrt{t-s}.
\end{alignat*}
These estimate prove that Assumption~\ref{assu:main} (B) and (C) are satisfied and thus the LDP from Theorem~\ref{thm:main} holds in~$C(\KK)$. Of course, this was already known since Freidlin and Wentzell~\cite{Freidlin70} with the seemingly different rate function
\[
I(\psi) = \inf\bigg\{ \half \int_0^1 u_t^2 \dt : u\in L^2([0,1]), \; \psi(t)= x + \int_0^t \sigma(\psi(s)) u_s \ds \bigg\}.
\]
Yet, Lemma 4.1.4 in \cite{DZ10} certifies that these two rate functions must be equal. While the representation above depends on an implicit ODE, \eqref{eq:ratefunction} enforces an explicit expansion for~$\psi$. This provides an alternative viewpoint on families of processes which LDP can be derived from both approaches.
\end{example}

%%%%%%%%%%%%%%%%%%%%%%%%%%%%%%%%%%%%%%%%%%%%%%%%%%%%%%%%%%%%%%%%%%%%
\subsection{Skorohod integrals}\label{subsec:SAE}

We introduce the divergence operator~$\delta$ as the adjoint of the Malliavin derivative~$\DD$, and refer to~\cite[Section 1.3]{Nualart06} for the details. In the white noise case, the domain of~$\delta$ is a subset of~$L^2(\Omega\colon L^2(\TT))$ and for~$Y\in\mathrm{Dom}(\delta)$, $\delta(Y)$ is called the Skorohod integral.
If, for all~$t\in\TT$, $Y_t$ has the chaos expansion~
$Y_t=\sum_{n=0}^\infty I_n(f_n^t)$, then the following expansion holds~\cite[Proposition 1.3.7]{Nualart06}:
\[
\delta(Y)=\sum_{n=1}^\infty I_n\big(\widetilde f_{n-1}\big),
\]
where~$\widetilde f_{n}$ is the symmetrisation of~$f_n^\cdot$:
\begin{equation*}\label{eq:kernelsym}
\widetilde f_n(\ttn,t) := \frac{1}{n+1} \left( f_n^t(\ttn) + \sum_{i=1}^n f_n^{t_i}(t_1,\cdots,t_{i-1},t,t_{i+1},\cdots,t_n)\right).
\end{equation*}
We note however that~$I_n\big(\widetilde f_{n-1}\big)=I_n\big(f_{n-1}^\cdot\big)$, while the symmetrisation allows us to recover our initial setting from subsection~\ref{subsec:chaos}.

Naturally, we can also consider multiple integrals of the form~$\delta^k (Y)$, for some~$k\in\bN$ and~$Y\in\mathrm{Dom}(\delta^k)\subset L^2(\Omega\colon L^2(\TT^k))$, where we define by iteration~$\delta^k(Y):=\delta(\delta^{k-1}(Y))$. 
We introduce the process~$\{X(z)\}_{z\in\KK}$ as~$X(z):=\delta^k(Y^z(W))$, for some process~$Y^z_{\ttk}(W):=\sum_{n=0}^\infty I_n(f_n^{z,\ttk})$ such that~$Y^z(W)\in \mathrm{Dom}(\delta^k)$, for all~$z\in\KK$.

For simplicity, we restrict this example to ``small-noise" large deviations, that means processes that can be written as~$\cG(\ep W)$ where~$\cG$ is a measurable map from~$\cM(\TT)$ to~$C(\KK)$, as the kernels of their expansion does not depend on~$\ep$ and they easily meet Assumption~\ref{assu:main}(A). One could of course study processes of the type~$\cG^\ep(\ep W)$ with kernels depending on~$\ep$.

Let us define the measurable map~$\cG\colon \cM(\TT)\to C(\KK)$ such that~$X=\cG(W)$, and the small-noise perturbation given by
\begin{equation}\label{eq:xepanticipative}
X^\ep(z):= \cG(\ep W)(z) = \ep^k I_k \big(\widetilde f_0^{\ep,z}\big) + \sum_{n=k+1}^\infty \ep^n I_n \big(\widetilde f_{n-k }^{z} \big), % = \sum_{n=0}^\infty \ep^{n+k} I_{n+k}(\widetilde f_{n}^z),
\end{equation}
where~$f_0^{\ep,z}(\ttk):=\bE\big[Y^z_{\ttk}(\ep W)\big]$ and the kernels require~$k$ successive symmetrisations.
The kernels of the white noise chaos expansion of~$X^\ep$ are thus explicit functions of the kernels of~$Y$. Therefore Theorem~\ref{thm:main} ensures that if~$\widetilde f_0^{\ep,z}$ converges to some~$\widetilde f_0^z$ and the family~$\{f_n^{z}\}_{z\in\KK,n\in\bN}$ satisfies Assumption~\ref{assu:main}(B,C) as functions of~$L^2(\TT^{n+k})$, then a pathwise LDP for~$\{X^\ep\}_{\ep>0}$ holds with rate function
\begin{align}
    \Lambda(\psi)& := \inf \bigg\{ \half \norm{u}_{L^2(\TT)}^2 : u\in L^2(\TT),\, \psi(z)= \sum_{n=k}^\infty J^u_n (\widetilde f^z_{n-k})  \bigg\} \nonumber\\
    &= \inf \bigg\{ \half \norm{u}_{L^2(\TT)}^2 : u\in L^2(\TT),\, \psi(z)= J^u_k \bigg(\sum_{n=0}^\infty J^u_n (\widetilde f^z_{n})\bigg)  \bigg\}.
\label{eq:rfMultiSkorohod}
\end{align}
Let us consider the case where~$\TT^k=\KK$, $\{f^{z}_{n}\}$ are already symmetric functions of~$L^2(\TT^{n+k})$ and $\{f^{z,\ttk}_{n}\}$~satisfy Assumption~\ref{assu:main}(B,C) as functions of~$L^2(\TT^{n})$.
Then for each~$z\in\KK$, $\{Y^z(\ep W)\}_{\ep>0}$ satisfies an LDP in~$C(\TT^k)$ with rate function
\[
\widehat\Lambda(\phi^z) := \inf \bigg\{ \half \norm{u}_{L^2(\TT)}^2 : u\in L^2(\TT),\, \phi^z(\ttk)= \sum_{n=0}^\infty J^u_n (f^{z,\ttk}_{n})  \bigg\}.
\]
The display~\eqref{eq:rfMultiSkorohod} leads to the new representation
\[
\Lambda(\psi) = \inf \Big\{ \widehat\Lambda(\phi^z): \phi^z\in C(\KK), \psi(z) = J^u_k(\phi^z) \Big\},
\]
which is reminiscent of the contraction principle and of~\cite[Theorem 2.1]{Garcia08} where conditions on the integrand yield an LDP for the stochastic integral.

%Note in particular that (C) ensures that~$X$ has continuous paths and that~$\cG$ is well-defined.

%\forest{In the particular case where~$Y^z$ is adapted to the filtration of the white noise process, one can relate the rate function to the multiple Skorohod integral. We make the dependence on $W$ explicit by writing~$Y^z(W)$; if~$\omega\mapsto Y^z(\omega)$ is continuous we deduce that~$\cG\left(\ep W^{u^\ep/\ep}\right)$ converges to the limiting process \[ X^u(z)= \int_{\TT^k} Y^z_{\ttk}\left(\int_\TT u(t)\,\widehat \mu(\D t)\right) \widehat\mu(t_1) u(t_1)\cdots u(t_k) \widehat \mu(\D t_k).\]}

%%%%%%%%%%%%%%%%%%%%%%%%%%%%%%%%%%%%%%%%%%%%%%%%%%%%%%%%%%%%%%%%%%%%%%%%%%%%%%%%%%%%%%
\subsection{Linear Skorohod equations}
We let~$\KK=\TT=[0,1]^d$ and consider for all~$\ep>0$ the small-noise equation
\begin{equation}
X^\ep(t)= x_0^t + \ep\delta(a^t X^\ep),
\label{eq:SkoEquation}
\end{equation}
where~$x_0^t\in\bR$ and~$a^t:\TT\to\bR$ for all~$t\in\TT$. % existence and uniqueness are proved by computing the kernels. 
Plugging in the putative chaos expansion of~$X$, the equation becomes
\begin{equation}\label{eq:putativeexpansion}
\sum_{n=0}^\infty \ep^n I_n(f_n^{t}) = x_0^t + \sum_{n=1}^\infty \ep^n I_n(g_{n-1}^t),
\end{equation}
where~$g_{n-1}^t(\ttn):=a^t(t_n)f_{n-1}^{t_n}(\bm{t_{n-1}})$ for all~$n\ge1$.
Ignoring the symmetrisation for now, we identify the Wiener chaos on both sides and deduce that~$f_0^t=x_0^t$ and~$f_n^t(\ttn)=g_{n-1}^t(\ttn)$ for all~$n\ge1$. Hence by induction we obtain
\begin{equation}\label{eq:SkoEqKernels}
f_n^t(\ttn)=a^t(t_n) a^{t_n}(t_{n-1})\cdots a^{t_2}(t_1) x_0^{t_1},
\end{equation}
which we can symmetrise in the following way:
\begin{equation}\label{eq:symmetrisation}
\widetilde f_n^t(\ttn) = \frac{1}{n!} \sum_\sigma f_n^t (t_{\sigma(1)},\cdots,t_{\sigma(n)}),
\end{equation}
with~$\sigma$ running over all permutations of~$\{1,\cdots,n\}$. This allows to recover the chaos expansion of our framework and give a proper meaning to the solution of~\eqref{eq:SkoEquation}. We note that the kernels are independent of~$\ep$ so that we can define the map~$\cG$ as:
\[
X^\ep:=\cG(\ep W)=\bE[\cG(\ep W)] + \sum_{n=1}^\infty \ep^n I_n(\widetilde f_n^t),
\]
Theorem~\ref{thm:main} yields explicit criteria for the LDP of~$\{X^\ep\}_{\ep>0}$ to hold with rate function~\eqref{eq:ratefunction}.

\begin{example}
To illustrate this, let us look at the special case~$\KK=\TT=[0,1]$ and~$a^t(s)=a^t \one_{s\le t}$ for some~$a^t\in\bR$. Then~\eqref{eq:SkoEqKernels} yields~$f_n^t(\ttn)=\one_{t_1\le t_2\le\cdots\le t_n\le t} \,a^t a^{t_n} \cdots a^{t_2} x_0^{t_1}$. We first notice that for almost all~$\ttn\in\TT^n$
\[
\Big\lvert \sum_\sigma \one_{t_{\sigma(1)}\le \cdots\le t_{\sigma(n)}\le t}\,a^t a^{t_{\sigma(n)}} \cdots a^{t_{\sigma(2)}} x_0^{t_{\sigma(1)}} \Big\lvert^2 
= \sum_\sigma  \one_{t_{\sigma(1)}\le \cdots\le t_{\sigma(n)}\le t}\,\abs{ a^t a^{t_{\sigma(n)}} \cdots a^{t_{\sigma(2)}} x_0^{t_{\sigma(1)}} }^2 
\]
because the cross terms vanish whenever~$t_1,\cdots,t_n$ are distinct. Hence from the symmetrisation~\eqref{eq:symmetrisation} we can compute explicitly
\begin{align*}
\norm{\widetilde f_n^{t}}^2_{L^2(\TT^n)}
&=\frac{1}{(n!)^2} \int_{\TT^n} \Big\lvert \sum_\sigma \one_{t_{\sigma(1)}\le t_{\sigma(2)}\le\cdots\le t_{\sigma(n)}\le t}\,a^t a^{t_{\sigma(n)}} \cdots a^{t_{\sigma(2)}} x_0^{t_{\sigma(1)} } \Big\lvert^2 \D \ttn  \ \\
&= \left(\frac{a^t}{n!}\right)^2 \sum_\sigma \int_0^t \int_0^{t_{\sigma(n)}} \cdots \int_0^{t_{\sigma(2)}}   \abs{a^{t_{\sigma(n)}} \cdots a^{t_{\sigma(2)}} x_0^{t_{\sigma(1)}} }^2 \,\D \ttn\\
&= \left(\frac{a^t}{n!}\right)^2 \sum_\sigma \frac{1}{n!}\int_{[0,t]^n} \abs{a^{t_{\sigma(n)}} \cdots a^{t_{\sigma(2)}} x_0^{t_{\sigma(1)}} }^2 \,\D\ttn\\
&= \left(\frac{a^t}{n!}\right)^2 \norm{a}^{2(n-1)}_{L^2[0,t]} \norm{x_0}^{2}_{L^2[0,t]},
\end{align*}
which yields Assumption~\ref{assu:main}(B) as the kernels are of exponential type~\eqref{eq:expotype}.
Analogously we see that
\[
\norm{\widetilde f_n^{t}-\widetilde f_n^{s}}_{L^2(\TT^n)} \le \abs{a^t-a^s}\frac{\norm{a}^{n-1}_{L^2(\TT)} \norm{x_0}_{L^2(\TT)}}{n!},
\]
which shows that the continuity criterion depends explicitly on the regularity of the map~$t\mapsto a^t$. For instance, if the latter is Hölder continuous then Assumption~\ref{assu:main}(C) is automatically satisfied.
\end{example}

One could generalise~\eqref{eq:SkoEquation} by considering multiple Skorohod integrals or taking~$a^t$ to belong to a given Wiener chaos and applying~\cite[Proposition 1.1.3]{Nualart06}.

%\forest{Furthermore, if~$X$ is adapted, the rate function is related to the following ODE:\[X^u(t) = \bE\left[\cG(0) \right] + \int_\TT a^t(s) X^u(s) u(s) \,\widehat \mu(\D s)\]Note that, in the case where~$\TT=[0,1]$, one can add a drift term of the form~$\int_0^t b X(s) \ds$ and also compute the kernels explicitly. }
%%%%%%%%%%%%%%%%%%%%%%%%%%%%%%%%%%%%%%%%%%%%%%%%%%%%%%%%%%%%%%%%%%%%%%%%%%%%%%%%%%%%%%
\subsection{Fractional integrals in the Wick--It\^o sense}\label{subsec:wick}
Let~$\KK=\TT=[0,1]$ and~$\widehat\mu$ be the Lebesgue measure, such that~$W$ is a standard Brownian motion.
We will use the Hida space~$(\cS)$ of stochastic test functions and its dual~$(\mathcal S)^*$ the Hida space of stochastic distributions; we refer to~\cite[Definition A.1.4]{Biagini08} for the definitions.
Consider the Hermite functions~$\{\xi_k\}_{k\ge1}$~\cite[Equation (A.6)]{Biagini08} and the operator~$\MM$~\cite[Equation (4.2)]{Biagini08} which, in particular, has the property~$\int_\bR \MM\one_{[0,t]}(x) \MM\one_{[0,s]}(x)\dx = R_H(t,s)$, the covariance function of the fractional Brownian motion. The latter was introduced in Example~\ref{ex:fbm} and we denote it by~$B^H$. The fractional white noise~$W^H$, not to be confused with the white noise \emph{measure}, is seen as the derivative of~$B^H$ in~$(\cS)^*$ and is defined in~\cite[Equation (4.16)]{Biagini08}:
\[
W^H(t) := \sum_{k=1}^\infty \MM\xi_k(t) I_1(\xi_k).
\]
Let~$Y:\TT\to L^2(\Omega)\subset(\cS)^\ast$ have the Wiener chaos expansion~$Y_t=\sum_{n=0}^\infty I_n(f_n^t)$ with respect to the standard Brownian motion. Then the Wick--It\^o integral with respect to the fractional Brownian motion~$B^H$ is defined as~\cite[Definition 4.2.2]{Biagini08}:
\[
\int_\TT Y_t \,\D B^H_t := \int_\TT Y_t \diamond W^H(t) \dt,
\]
where~$\diamond$ denotes the Wick product.
%If~$Y$ is Skorohod integrable then this integral coincides with~$\delta_H(Y)$.
We then notice by~\cite[Theorem A.3.7]{Biagini08} and~\cite[Proposition 1.1.2]{Nualart06} that for all~$f\in L^2(\TT^n)$ and~$g\in L^2(\TT)$:
\begin{align*}
I_n(f) \diamond I_1(g) &= I_n(f) I_1(g) - \int_\TT g(t) \DD_t I_n(f) \dt \\
&= I_{n+1}(f\otimes g) + n I_{n-1}(f\otimes_1 g) - \int_\TT g(t) \,n I_{n-1}\big(f(\cdot,t)\big) \dt,
\end{align*}
where~$(f\otimes_1 g)(\bm{t_{n-1}}):=\int_\TT f(\bm{t_{n-1}},t) g(t) \dt$. Therefore we obtain~$I_n(f) \diamond I_1(g) =I_{n+1}(f\otimes g)$. By linearity this yields, if the series are convergent:
\begin{align}\label{eq:WickChaos}
    \int_\TT Y_s \diamond W^H(s)\dt
    = \int_\TT \, \sum_{\substack{k=1\\n=0}}^\infty \Big(I_{n+1}\big(f_n^s \otimes \xi_k\big) \MM\xi_k(s) \Big) \ds
    = \sum_{n=1}^\infty I_n(g_n),
\end{align}
where~$g_n:= \int_\TT\sum_{k=1}^\infty \big(f_{n-1}^s \otimes \xi_k\big) \MM\xi_k(s) \ds$.
One can symmetrise the kernels as in~\eqref{eq:symmetrisation} to recover the white noise chaos expansion.

Let us now introduce the family of processes~$X^\ep(t)=\ep \int_\TT Y^t_s(\ep W) \,\D B^H_s$ for all~$\ep>0$, where the integral is taken in the Wick--It\^o sense and~$Y^t_s(W):=\sum_{n=0}^\infty I_n(f_n^{s,t})$. There exists a measurable map~$\cG:\cM(\TT)\to\bR$ such that~$\cG(\ep W):=X^\ep$ and similarly to~\eqref{eq:WickChaos} we find that
\[
X^\ep(t)= \ep I_1 \big(g_1^{\ep,t}\big) + \sum_{n=2}^\infty \ep^n I_n\left(g_n^t\right),
\]
where~$g_1^{\ep,t}:=\int_\TT\sum_{k=1}^\infty \big(\bE[Y^t_s(\ep W)]\otimes \xi_k\big) \MM\xi_k(s) \ds$, and $g_n^t:= \int_\TT\sum_{k=1}^\infty \big(f_{n-1}^{s,t} \otimes \xi_k\big) \MM\xi_k(s) \ds$ for all~$n\ge2$. After symmetrisation of the kernels,
sufficient conditions for the large deviations of~$\{X^\ep\}_{\ep>0}$ follow from Theorem~\ref{thm:main} and can be reduced to estimates on~$\lVert f_{n-1}^{s,t} \otimes \xi_k \lVert_{L^2(\TT^n)}$ and $\lVert(f_{n-1}^{s,t_1}-f_{n-1}^{s,t_2}) \otimes \xi_k\lVert_{L^2(\TT^n)}$. These conditions imply in particular that the series are convergent and that~$X^\ep$ has continuous paths.

Similarly to the previous section we can consider fractional equations of the type
\[
X^\ep(t) = x_0^t + \ep \int_\TT a^t(s) X^\ep(s) \,\D B^H_s,
\]
where~$a$ is deterministic, and identify the kernels of the chaos expansion as in~\eqref{eq:putativeexpansion}. They satisfy~$f_0^t=x_0^t$ and
\[
f_n^t(\ttn) = \int_\TT a^t(s) f_{n-1}^s(\bm{t_{n-1}}) \sum_{k=1}^\infty \xi_k(t_n)\MM\xi_k(s) \ds, \quad \mathrm{for\; all }\; n\ge1.
\]
Even though the expression is more intricate than in the white noise case~\eqref{eq:SkoEqKernels}, after symmetrisation this also yields sufficient conditions for the pathwise LDP of the small-noise version of~$X$.
%%%%%%%%%%%%%%%%%%%%%%%%%%%%%%%%%%%%%%%%%%%%%%%%%%%%%%%%%%%%%%%%%%%%%%%%%%%%%%%%%%%%%%%%%%%%%%%%%%%%%%%%%%%%%%%%%%%%%%%%%%%%%%%%%%%%%%%%%%
\subsection{Martingales} \label{subsec:martingale}
We come back to the general case for~$\KK,\TT$ and~$\widehat\mu$ and illustrate the regularity condition~\eqref{eq:equicont} through an example where the $z$-dependence of the kernels is more explicit, and show the connection with closed martingales.
Consider the case where the kernels have the form~$f_n^{\ep,z} = \mathrm f_n^\ep \, \mathfrak{f}_n^z$, for some functions~$\mathrm f_n^\ep,\,\mathfrak{f}_n^z\colon\TT^n\to\bR$ and for all~$\ep>0,\,z\in\KK$ and~$n\ge0$.
Then, for all~$q>2$, we can write by H\"older's inequality
\[
\norm{f_n^{\ep,z}- f_n^{\ep,y}}_{L^2(\TT^n)}
\le \norm{\mathrm f_n^\ep}_{L^q(\TT^n)} \norm{\mathfrak{f}_n^z-\mathfrak{f}_n^y}_{L^\frac{2q}{q-2}(\TT^n)},
%\mathbb a, \mathbb b, \mathbb c, \mathbb d, \mathbb e, \mathbb f, \mathbb g, \mathbb h, \mathbb i
\]
which allows to split the responsibilities of Assumption~\ref{assu:main} (B) and (C). Indeed, both are satisfied if there exist~$q>2$ and~$\Delta>0$ such that, for all~$n$ large enough, all~$\ep$ small enough and all~$y,z\in\KK$,
\[
\norm{\mathrm f_n^\ep}_{L^q(\TT^n)} \le \frac{\Delta^n}{n!} \quad \mathrm{and} \quad 
\norm{\mathfrak{f}_n^z-\mathfrak{f}_n^y}_{L^\frac{2q}{q-2}(\TT^n)} \le \omega(\norm{z-y}).
\]

For all~$A\in\cT$, let~$\cF_A$ be the~$\sigma$-field generated by~$\{\dot W(B):B\subset A,\,B\in\cT\}$. Moreover, assume there exists an~$L^2(\Omega)$-random variable~$\widetilde X^\ep=\sum_{n=0}^\infty \ep^n I_n(f_n^\ep)$ such that, for all~$z\in\KK$, there exists~$A_z\in\cT$ such that~$X^\ep(z)=\bE \big[ \widetilde X^\ep \lvert \cF_{A_z} \big]$. Then, Lemma 1.2.5 in~\cite{Nualart06} entails that
\[
X^\ep(z) = \sum_{n=0}^\infty \ep^n I_n \big(f_n^\ep \,\one_{A_z}^{\otimes n} \big).
\]
\begin{example}
Let~$\KK=\TT=[0,1]^d$ and~$A_z = [0,z]:=\{y\in\TT:y(i)\le z(i) \;\mathrm{ for\; all }\;i\}$. Then, for~$y,z\in\KK$ such that~$\norm{z-y}\le 1$, we have
\[
\norm{\mathfrak{f}_n^z-\mathfrak{f}_n^y}_{L^\frac{2q}{q-2}(\TT^n)} = \norm{\one_{A_z}^{\otimes n}-\one_{A_y}^{\otimes n}}_{L^\frac{2q}{q-2}(\TT^n)} 
= \left( \prod_{i=1}^d \abs{z(i)-y(i)} \right)^{n\frac{q-2}{2q}} 
 \le \norm{z-y}^{\frac{q-2}{2q}}.
\]
If~$d=1$, $X^\ep$ is a closed martingale and this type of kernel corresponds to the ``adapted case" in~\cite{MWNPA92}.
\end{example}
%%%%%%%%%%%%%%%%%%%%%%%%%%%%%%%%%%%%%%%%%%%%%%%%%%%%%%%%%%%%%%%%%%%%%%%%%%%%%%%%%%%%%%%

%%%%%%%%%%%%%%%%%%%%%%%%%%%%%%%%%%%%%%%%%%%%%%%%%%%%%%%%%%%%%%%%%%%%%%%%%%%%%%%%%%%%%%%
\section{Conclusion}\label{sec:ccl}
This paper provides a pathwise large deviations principle for Wiener chaos expansions, and thus establishes an intermediate level of generality between the abstract machinery of Budhiraja and Dupuis~\cite{BD01} and specific models such as stochastic equations. 
We also show that this approach embraces and sheds new light on the asymptotic behaviour of Gaussian functionals that can be examined through the chaos expansion lens.
Finally, we believe it opens up room for future research both horizontally---for different types of driving noise, and vertically---for exploring the scope of applications.

\bibliographystyle{alpha}
\bibliography{bib}

\end{document}